\documentclass{amsart}

\usepackage[utf8]{inputenc}
\usepackage[T1]{fontenc}

\usepackage{amsmath, amssymb, amsthm, mathtools}
\usepackage{mathrsfs, mathpazo}
\usepackage{physics}
\usepackage{geometry}

\usepackage{float}
\usepackage[all]{xy}
\usepackage{graphicx}
\usepackage{color}
\usepackage{tikz}
\usepackage{tikz-cd}
\usepackage{xparse}
\usetikzlibrary{matrix,arrows,decorations.pathmorphing}

\usepackage{cite}
\usepackage{fancyhdr}
\usepackage{hyperref}
\usepackage{cleveref}

\usepackage{indentfirst}
\usepackage{enumerate}
\usepackage{enumitem}
\usepackage{wrapfig}
\usepackage{adjustbox}
\usepackage{leftidx}
\usepackage{stackrel}
\usepackage{comment}

\newtheorem*{theorem*}{Theorem}

\newtheorem{theorem}{Theorem}[section]

\newtheorem{corollary}{Corollary}[section]

\newtheorem{lemma}{Lemma}[section]

\newtheorem{conjecture}{Conjecture}[section]
\newtheorem{problem}{Problem}[section]

\crefname{theorem}{Theorem}{theorem}
\crefname{lemma}{Lemma}{lemma}
\crefname{remark}{Remark}{remark}
\crefname{corollary}{Corollary}{corollary}
\crefname{proposition}{Proposition}{proposition}
\crefname{example}{Example}{example}
\crefname{definition}{Definition}{definition}
\crefname{notation}{Notation}{notation}
\crefname{appendix}{Appendix}{appendix}
\crefname{section}{Section}{section}
\crefname{question}{Question}{question}

\begin{document}
\title{On arithmetic properties of Cantor sets}

\author{Lu Cui}
\address{Institute of Mathematics, Academy of Mathematics and Systems Science, Chinese Academy of Sciences, Beijing 100190, China}
\email{cuilu17@mails.ucas.ac.cn}
\thanks{}

\author{Minghui MA}
\address{Institute of Mathematics, Academy of Mathematics and Systems Science, Chinese Academy of Sciences, Beijing 100190, China}
\email{maminghui17@mails.ucas.ac.cn}
\thanks{}


\begin{abstract}
    Three types of Cantor sets are studied.
For any integer $m\ge 4$, we show that every real number in $[0,k]$ is the sum of at most $k$ $m$-th powers of elements in the Cantor ternary set $C$ for some positive integer $k$, and the smallest such $k$ is $2^m$.
Moreover, we generalize this result to middle-$\frac 1\alpha$ Cantor set for  $1<\alpha<2+\sqrt{5}$ and $m$ sufficiently large.
For the naturally embedded image $W$ of the Cantor dust $C\times C$ into the complex plane $\mathbb{C}$, we prove that for any integer $m\ge 3$, every element in the closed unit disk in $\mathbb C$ can be written as the sum of at most $2^{m+8}$ $m$-th powers of elements in $W$.
At last, some similar results on $p$-adic Cantor sets are also obtained.
\end{abstract}

\maketitle

\section{Introduction}

The classical Waring's problem in number theory asks whether for any integer $m\ge 3$, there is an associated positive integer $k$ such that every natural number can be written as the sum of at most $k$ $m$-th powers of natural numbers.
This problem and its variations have received much attention (see, for example, \cite{Zhao2,Liu1,Liu2,Vaughan,Wu,Zhao1}).

Recently, Guo studied analogs of Waring's problem on Cantor sets in \cite{Guo}, which was called the Waring-Hilbert problem.
Let
\[
C=\left\{\sum_{n=1}^\infty\frac{a_n}{3^n}: a_n\in\{0,2\}\right\}\subseteq[0,1]
\]
be the Cantor ternary set.
In 1917, Steinhaus proved in \cite{Steinhaus} that
\begin{equation*}
  C+C=\{x+y: x,y\in C\}=[0,2].
\end{equation*}
In 2019, Athreya, Reznick and Tyson studied arithmetic of Cantor set (see \cite{Athreya}) and conjectured that every element in $[0,1]$ can be written as $x_1^2+x_2^2+x_3^2+x_4^2$ with each $x_j$ in $C$.
This conjecture was proved by Wang, Jiang, Li and Zhao in \cite{Wang}.
In fact, they considered general middle-$\frac 1\alpha$ Cantor sets $C_\alpha$ obtained by removing (successively) the middle one open interval with length $\frac{1}{\alpha}$ of the original for $\alpha>1$ and showed that
\[
[0,4]=\{x_1^2+x_2^2+x_3^2+x_4^2:x_j\in C_\alpha, 1\le j\le 4\}
\]
if and only if $\alpha\ge 3$.

After that, Guo proved that, for any integer $m\ge 1$, every real number in $[0,1]$ can be written as $x_1^m+x_2^m+\cdots+x_k^m$ with each $x_j$ in $C$ for some $k\le 6^m$ (see \cite{Guo}).
He conjectured that
\begin{equation}\label{Guo1}
  [0,2^m]=\{x_1^m+x_2^m+\cdots+x_{2^m}^m:x_j\in C,1\le j\le 2^m\}
\end{equation}
for each integer $m\ge 4$.
He proved (\ref{Guo1}) for $m=3$.
For the middle-$\frac{1}{2}$ Cantor set $C_2$, he conjectured that
\begin{equation}\label{Guo2}
  [0,3^m]=\{x_1^m+x_2^m+\cdots+x_{3^m}^m:x_j\in C_2,1\le j\le 3^m\}
\end{equation}
for each integer $m\ge 1$.

The main purpose of this article is to solve the problems proposed by Guo.
In Section 2, we introduce some notation and prove some auxiliary results.
For basics of number theory, set theory, and real analysis, see \cite{Hua,Pan,Rudin,Ye}.

In Section 3, denote by $G_\alpha(m)$ the smallest positive integer $k$ that satisfies
\begin{equation*}
[0,k]=\{x_1^m+x_2^m+\cdots+x_k^m:x_j\in C_\alpha,1\le j\le k \}.
\end{equation*}
We provide an upper bound of $G_\alpha(m)$ for any real number $m\ge 1$ and $\alpha>1$ (see Theorem \ref{lemupperbound}).
Let $r=\frac{1}{2}(1-\frac{1}{\alpha})$.
Note that if $k<(\frac{1}{r}-1)^m$, then elements in $(kr^m,(1-r)^m)$ can not be written as $x_1^m+x_2^m+\cdots+x_k^m$ with each $x_j$ in $C_\alpha$ (see Lemma \ref{thmlowerbound}).
We shall prove that, for any $\alpha$ in $(1,2+\sqrt{5})$, if $m$ is sufficiently large, then $G_\alpha(m)=\left\lceil\left(\frac{1}{r}-1\right)^m\right\rceil$ (see Theorem \ref{upperbound}).

In Section 4, by using the result obtained in Section 3, we show that (\ref{Guo1}) holds for each integer $m\ge4$ (see Theorem 4.1) and (\ref{Guo2}) holds for each integer $m\ge7$ (see Theorem 4.2).

In Section 5, we study another two questions raised by Guo.
In \cite[Section 3]{Guo}, Guo asked whether for any $m\ge 1$, there is an $\varepsilon>0$ such that
\begin{equation}\label{Guo3}
  [2-\varepsilon,2]\subseteq\{x_1^m+x_2^m:x_1,x_2\in C\}.
\end{equation}
He speculated that, for any $m\ge 1$, there exists an $\epsilon>0$, such that
\begin{equation}\label{Guo4}
  [3-\epsilon,3]\subseteq\{x_1^m+x_2^m+x_3^m:x_1,x_2,x_3\in C\}.
\end{equation}
We prove question (\ref{Guo4}) (see Theorem 5.2) and give a negative answer to question (\ref{Guo3}) (see Theorem 5.1).

In Section 6, we prove that for each integer $m\ge 3$, there is a positive integer $k\le 2^{m+8}$, such that
\begin{equation*}
  \{z: z\in\mathbb{C}, |z|\le 1\}\subseteq\{z_1^m+ \cdots + z_k^m: z_j\in W,1\le j\le k\},
\end{equation*}
where $W=\{x+iy: x,y\in C\}\cong C\times C$ (see Theorem 6.2).
This is another conjecture of Guo (see \cite[Section 6]{Guo}).

In Section 7, we introduce a class of Cantor sets in the rings of $p$-adic integers $\mathbb Z_p\subseteq \mathbb Q_p$ for every prime $p$, and study generalizations of Waring's problem on these Cantor sets.

\section{Notation and preliminaries}

Let $\mathbb{Q}$, $\mathbb{Z}$, $\mathbb{N}$, $\mathbb{R}$ and $\mathbb{R}_+$ denote the set of rational numbers, integers, non-negative integers, real numbers and positive real numbers, respectively.
Let $n,N,$ and $k$ be integers throughout this article.
For non-empty sets $X$ and $Y$ in the complex plane $\mathbb{C}$ and $t\in \mathbb{C}$, we adopt the notation
\begin{equation*}
  X+Y=\{x+y: x\in X,y\in Y\},
\end{equation*}
\begin{equation*}
  tX=\{tx: x\in X\},\  t+X=\{t+x: x\in X\},
\end{equation*}
\begin{equation*}
X^{k}=\underbrace{X\times X\times \cdots \times X}_k=\big\{(x_{1},x_{2},\cdots,x_k):x_{1},x_{2},\cdots,x_k\in X\big\}.
\end{equation*}
Let $\alpha>1$ be a real number and $r=\frac{1}{2}(1-\frac{1}{\alpha})\in(0,\frac{1}{2})$.
Two contractive maps $g_0$ and $g_1$ are defined on $[0,1]$ by
\begin{equation*}
g_0(x)=rx\ \ \text{and}\ \  g_1(x)=1-r+rx.
\end{equation*}
For $n\geq 1$, define
\begin{equation*}
\mathcal{F}_{n}=\big\{g_{\sigma}([0,1]):\sigma\in\{0,1\}^{n}\big\}\ \ \text{and}\ \ F_{n}=\bigcup\limits_{A\in\mathcal{F}_{n}}A,
\end{equation*}
where $g_{\sigma}(x)=g_{\sigma_{1}}\circ g_{\sigma_{2}}\circ\cdots\circ g_{\sigma_{n}}(x)$ for $\sigma=\sigma_{1}\sigma_{2}\cdots\sigma_{n}\in\{0,1\}^{n}$.
Then every $F_{n}$ is a union of finite closed intervals and $F_{n}\supseteq F_{n+1}$.
Moreover, the middle-$\frac{1}{\alpha}$ Cantor set is $C_\alpha=\bigcap\limits_{n=1}^{\infty}F_{n}$.
The Cantor ternary set $C$ is $C_3$.

Let $L_{n}=\big\{g_{\sigma}(0):\sigma\in\{0,1\}^{n}\big\}$ be the set of all left endpoints of intervals in $\mathcal{F}_{n}$.
For any $u\in L_{n}$, we write
\begin{equation*}
I_{u}=[u,u+r^{n}]\in\mathcal{F}_{n}.
\end{equation*}
Note that $I_{u}$ contains two intervals $I_{u,0}$ and $I_{u,1}$ in $\mathcal{F}_{n+1}$.
We write
\begin{equation*}
I_{u,0}=[u,u+r^{n+1}]\ \ \text{and} \ \
I_{u,1}=[u+(1-r)r^{n},u+r^{n}].
\end{equation*}
For real number $m\ge 1$ and $(x_{1},x_{2},\cdots,x_{k})\in \mathbb{R}_+^k$ or integer $m\ge 1$ and $(x_{1},x_{2},\cdots,x_{k})\in \mathbb{C}^k$, define
\begin{equation*}
  f_{k,m}(x_{1},x_{2},\cdots,x_{k})=x_{1}^{m}+x_{2}^{m}+\cdots+x_{k}^{m}.
\end{equation*}
For real number $m\ge 1$, let $G_\alpha(m)$ denote the smallest positive integer $k$ such that $[0,k]=f_{k,m}(C_\alpha^k)$, and let $G(m)=G_3(m)$.
For $a,x\in\mathbb{R}_+$ and $a\neq 1$, we denote by  $\log x$ the natural logarithm of $x$ with base $e$ and $\log_a x=\frac{\log x}{\log a}$.
For $x\in\mathbb{R}$, let
\begin{equation*}
  \lfloor x\rfloor=\max\{k\in\mathbb{Z}, k\le x\}\ \ \text{and}\ \
  \lceil x\rceil=\min\{k\in\mathbb{Z}, k\ge x\}.
\end{equation*}
If $X$ is a non-empty compact subset of $\mathbb{R}$, then $[\inf X,\sup X]\setminus X=\bigcup\limits_{n\in\mathbb{N}}(a_n,b_n)$ is a union of disjoint open intervals.
Define
\begin{equation*}
  \text{Gap}(X)=\sup_n(b_n-a_n).
\end{equation*}

\begin{lemma}\label{lemgap}
  If $X$ is a non-empty compact subset of $\mathbb{R}$ and $0<\text{Gap}(X)\leq b-a$, then $X+[a,b]$ is an interval.
\end{lemma}

\begin{proof}
Let $L=\inf X$, $R=\sup X$.
We only need to prove $X+[a,b]=[L+a,R+b]$.
It is obvious that $X+[a,b]\subseteq[L+a,R+b]$, so it suffices to show $[L+a,R+b]\subseteq X+[a,b]$.

Let $y\in[L+a,R+b]$.
We have $y-b\leq R$ and $y-a\ge L$.
\begin{enumerate}
  \item If $y-b\leq L$, then $L\in[y-b,y-a]$.
  \item If $y-a\ge R$, then $R\in[y-b,y-a]$.
  \item If $L<y-b<y-a<R$, then $[y-b,y-a]\subseteq [L,R]$.
Notice that $\text{Gap}(X)\leq b-a$.
We have $[y-b,y-a]\cap X\neq \emptyset$.
\end{enumerate}
From the discussion of the above three cases, we derive $[y-b,y-a]\cap X\neq \emptyset$.
Choosing $x\in [y-b,y-a]\cap X$, we obtain $y-x\in[a,b]$ and $y=x+(y-x)\in X+[a,b]$.
The proof is complete.
\end{proof}

Corollary 2.2.1 and Corollary 2.3.1 of \cite{Guo} are used frequently in our Section 3.
For completeness, we state and prove them as follows.

\begin{lemma}\label{fkmCk}
  For real number $m\ge 1$, $\alpha>1$ and $k\in\mathbb{N}$, we have $f_{k,m}(C_\alpha^{k})=\bigcap\limits_{n=1}^{\infty} f_{k,m}(F_n^{k})$.
\end{lemma}

\begin{proof}
This lemma is a special case of the following result.
Let $f$ be a real valued continuous function on $\mathbb{R}^k$ with $k\geq 1$.
If $\{X_n\}_{n=1}^{\infty}$ is a decreasing sequence of non-empty compact subsets of $\mathbb{R}^k$, then $f\left(\bigcap\limits_{n=1}^{\infty}X_n\right)=\bigcap\limits_{n=1}^{\infty}f(X_n)$.
We now prove this general case.

On the one hand, we have $\bigcap\limits_{n=1}^{\infty}X_n \subseteq X_l$ for any positive integer $l$, then  $f\left(\bigcap\limits_{n=1}^{\infty}X_n\right)\subseteq\bigcap\limits_{l=1}^{\infty}f(X_l)$.

On the other hand, let $y\in\bigcap\limits_{n=1}^{\infty}f(X_n)$.
There is an $x_n\in X_n$ such that $f(x_n)=y$ for $n\geq 1$.
Since $\{X_n\}_{n=1}^{\infty}$ is decreasing and $X_1$ is compact, we obtain that $\{x_n\}_{n=1}^{\infty}\subseteq X_1$ has a convergent subsequence $\{x_{n_j}\}_{j=1}^{\infty}$, which converges to some $x$ in $X_1$.
We have $f(x)=y$ from the continuity of $f$.
Note that $\{x_{n_j}\}_{j\geq l}\subseteq X_l$ and $X_l$ is compact for any positive integer $l$.
We obtain $x\in X_l$ and $x\in \bigcap\limits_{l=1}^{\infty}X_l$.
It follows that $y=f(x)\in f\left(\bigcap\limits_{l=1}^{\infty}X_l\right)$ and $\bigcap\limits_{n=1}^{\infty}f(X_n)\subseteq f\left(\bigcap\limits_{l=1}^{\infty}X_l\right)$.
The proof is complete.
\end{proof}

\begin{lemma}\label{lemrk}
  For real number $m\ge 1$, $\alpha>1$ and $k\in\mathbb{N}$, if $X\subseteq f_{k,m}(C_\alpha^{k})$, then $r^{lm} X\subseteq f_{k,m}(C_\alpha^{k})$ for any $l\in\mathbb{N}$.
\end{lemma}

\begin{proof}
Let $y\in X\subseteq f_{k,m}(C_\alpha^{k})$.
Then $y=f_{k,m}(x_1,x_2,\cdots,x_k)$ for some  $(x_1,x_2,\cdots,x_k)\in C_\alpha^{k}$.
Since $x\in C_\alpha$ implies $rx\in C_\alpha$, we obtain $(rx_1,rx_2,\cdots,rx_k)\in C_\alpha^{k}$.
It follows from the definition of $f_{k,m}$ that
\begin{equation*}
r^m y=r^m f_{k,m}(x_1,x_2,\cdots,x_k)=f_{k,m}(rx_1,rx_2,\cdots,rx_k)\in f_{k,m}(C_\alpha^{k}).
\end{equation*}
By induction on $l$, we derive $r^{lm}y\in f_{k,m}(C_\alpha^{k})$ for $l\in\mathbb{N}$.
Hence, we obtain $r^{lm} X\subseteq f_{k,m}(C_\alpha^{k})$.
\end{proof}

Some inequalities are used frequently in this article, we state and prove them for completeness.

\begin{lemma}\label{inequality}
Suppose that $t>1$ and $m\geq 1$.
\begin{enumerate}
  \item For $x\ge 0$, we have $\arctan(tx) \leq t\arctan x$.
  \item For $x\in[0,1]$, we have $\arctan x \leq x \leq\frac{4}{\pi}\arctan x$, $x\leq\frac{\pi}{2}\sin x$ and
     \begin{equation*}
    \frac{5}{4}\arctan x\leq \arctan\left(\frac{3}{2} x\right)\le\frac{3}{2}\arctan x\leq \arctan (3x)\leq 3\arctan x-x^3.
  \end{equation*}
  \item For $x\in[0,\frac{1}{3}]$, we have $\frac{7}{5}\arctan x\leq \arctan(\frac{3}{2} x)$ and $\frac{9}{5}\arctan x \leq \arctan (2x)\leq 2\arctan x-x^3$.
  \item For $x\in [0,1)$, we have $\ln(1-x)\geq x\ln(1-x)-x$.
  \item For $x\in[0,\frac{1}{t}]$, we have $(1-x)^m-(1-tx)^m\leq 1-\frac{1}{t}$.
\end{enumerate}
\end{lemma}

\begin{proof}
Let $h_t(y)=\arctan(t\tan y)$ for $y\in[0,\frac{\pi}{2})$.
Then $h_t''(y)=-\frac{2t(t^2-1)(\text{sec} y)^2\tan y}{(1+(t\tan y)^2)^2}\le 0$.
From the concavity of $h_t(y)$ and $h_t(0)=0$, we derive $\frac{h_t(y)}{y}$ is decreasing.
For $t>1$, define
\begin{equation}\label{1}
  \varphi_t(x)=\frac{\arctan(tx)}{\arctan x}.
\end{equation}
Then $\varphi_t(x)$ is decreasing for $x>0$ by letting $y=\arctan x$.
\begin{enumerate}
  \item Since $\varphi_t(x)\le\lim\limits_{y\to 0+}\varphi_t(y)=t$, the first inequality holds.
  \item Since $\arctan x$ is concave for $x\in [0,1]$ and $\arctan 0=0$, the function $\frac{\arctan x}{x}$ is decreasing, then it is easy to see that $\arctan x \leq x \leq\frac{4}{\pi}\arctan x$.
      Similarly, we have $\sin x\le x\le \frac{\pi}{2}\sin x$ for $x\in[0,\frac{\pi}{2}]$.

      For $x\in [0,1]$, we derive $\varphi_{\frac{3}{2}}(x)\ge\varphi_{\frac{3}{2}}(1)>\frac{5}{4}$ and $\varphi_3(x)\ge\varphi_3(1)>\frac{3}{2}$ by letting $t=\frac{3}{2}$ and $t=3$ in (\ref{1}), respectively. Then $\frac{5}{4}\arctan x\leq \arctan(\frac{3}{2} x)$ and $\frac{3}{2}\arctan x\leq \arctan(3x)$.

      Let $\psi_1(x)=3\arctan x-\arctan (3x)-x^3$.
      Then $\psi_1'(x)=3x^2\cdot \frac{7-(10x^2+9x^4)}{(1+x^2)(1+9x^2)}$ and it is not hard to see that $\psi_1(x)$ increases firstly and then decreases for $x\in[0,1]$.
      Note that $\psi_1(0)=0$ and $\psi_1(1)=\frac{3\pi}{4}-\arctan 3-1>0$.
      We derive $\psi_1(x)\geq 0$ and $\arctan (3x)\leq 3\arctan x-x^3$ for $x\in[0,1]$.
  \item For $x\in [0,\frac{1}{3}]$, we derive  $\varphi_{\frac{3}{2}}(x)\ge\varphi_{\frac{3}{2}}(\frac{1}{3})>\frac{7}{5}$
      and $\varphi_{2}(x)\ge\varphi_{2}(\frac{1}{3})>\frac{9}{5}$ by letting $t=\frac{3}{2}$ and $t=2$ in (\ref{1}), respectively.
      Then $\frac{7}{5}\arctan x\leq \arctan(\frac{3}{2} x)$ and $\frac{9}{5}\arctan x\leq \arctan(2x)$.

      Let $\psi_2(x)=2\arctan x-\arctan(2x)-x^3$.
      Then $\psi_2'(x)=3x^2\cdot \frac{1-(5x^2+4x^4)}{(1+x^2)(1+4x^2)}\geq 0$ for $x\in[0,\frac{1}{3}]$.
      It follows that $\psi_2(x)\geq\psi_2(0)=0$ and $\arctan (2x)\leq 2\arctan x-x^3$ for $x\in\left[0,\frac{1}{3}\right]$.
  \item Using the fact that $\left(\ln(1-x)-x\ln(1-x)+x\right)'=-\ln(1-x)>0$ for $0\le x<1$, we obtain  $\ln(1-x)\geq x\ln(1-x)-x$ for $0\le x<1$.
  \item It is clear that if $m=1$, then the last inequality in the lemma is valid.
      We next assume $m>1$.
      Let $h(x)=(1-x)^m-(1-tx)^m$.
      Then $h'(x)=-m(1-x)^{m-1}+tm(1-tx)^{m-1}$.
      Let $h'(x)=0$, we have $x=\frac{a-1}{ta-1}$, where $a=t^{\frac{1}{m-1}}$ and $a^m=ta$.
      Moreover, we obtain $h'(x)\geq 0$ for $x\in [0,\frac{a-1}{at-1}]$ and $h'(x)\leq 0$ for $x\in [\frac{a-1}{at-1},\frac{1}{t}]$.
Since $t>1$, we have $a>1$ and $\frac{t-1}{ta-1}\le\frac{1}{a}$.
Therefore, we obtain
  \begin{equation*}
    h(x)\le h\left(\frac{a-1}{ta-1}\right)=\frac{(a^m-1)(t-1)^m}{(ta-1)^m}
    =(t-1)\left(\frac{t-1}{ta-1}\right)^{m-1}\leq \frac{t-1}{a^{m-1}}=1-\frac{1}{t}.
  \end{equation*}
The last inequality in the lemma holds.
\end{enumerate}
\end{proof}

\section{Sum of $m$-th powers of elements in $C_\alpha$}

Let $m\geq 1$ be a real number throughout this section.

The next two lemmas improve the Lemma 2.3 and Lemma 2.4 of \cite{Wang} and give a criterion for finding some intervals in $f_{k,m}(C_\alpha^k)$.

\begin{lemma}
Suppose that $k\in\mathbb{N}$, $k\geq 2$ and $u_{1},u_{2},\cdots,u_{k}\in L_{n}$.
If
\begin{equation*}
  \sum\limits_{i\neq M}u_i^{m-1} \geq \lambda(u_M+r^n-r^{n+1})^{m-1},
\end{equation*}
where $\lambda=\frac{1}{r}-2$ and $u_M=\max\{u_{1},u_{2},\cdots,u_{k}\}$, then
\begin{equation*}
f_{k,m}(I_{u_{1}}\times I_{u_{2}}\times\cdots\times I_{u_{k}})=f_{k,m}\left((I_{u_1,0}\cup I_{u_1,1})\times (I_{u_2,0}\cup I_{u_2,1})\times\cdots\times (I_{u_k,0}\cup I_{u_k,1})\right).
\end{equation*}
\end{lemma}

\begin{proof}
It is obvious that the right-hand side is a subset of the left-hand side since $I_{u_i,0}\cup I_{u_i,1}\subseteq I_{u_i}$ for $i=1,2,\cdots, k$.
Hence, we only need to prove
\begin{equation*}
f_{k,m}(I_{u_{1}}\times I_{u_{2}}\times\cdots\times I_{u_{k}})\subseteq f_{k,m}\left((I_{u_1,0}\cup I_{u_1,1})\times (I_{u_2,0}\cup I_{u_2,1})\times\cdots\times (I_{u_k,0}\cup I_{u_k,1})\right).
\end{equation*}

Let $Q_{v}=f_{k,m}(I_{u_1,v_1}\times I_{u_2,v_2}\times\cdots\times I_{u_k,v_k})$ for $v=v_{1}v_{2}\cdots v_{k}\in\{0,1\}^{k}$.
Then we have
\begin{equation*}
f_{k,m}\left((I_{u_1,0}\cup I_{u_1,1})\times (I_{u_2,0}\cup I_{u_2,1})\times\cdots\times (I_{u_k,0}\cup I_{u_k,1})\right)=\bigcup\limits_{v\in\{0,1\}^{k}}Q_{v}.
\end{equation*}
Note that
\begin{equation*}
  f_{k,m}(I_{u_{1}}\times I_{u_{2}}\times\cdots\times I_{u_{k}})=\left[\sum\limits_{i=1}^{k}u_i^m, \sum\limits_{i=1}^{k}(u_i+r^n)^m\right]
\end{equation*}
is an interval with the same minimum $\sum\limits_{i=1}^{k}u_i^m$ and maximum $\sum\limits_{i=1}^{k}(u_i+r^n)^m$ as $\bigcup\limits_{v\in\{0,1\}^{k}}Q_{v}$.
It suffices to show that $\bigcup\limits_{v\in\{0,1\}^{k}}Q_{v}$ is connected.

For any $v=v_{1}v_{2}\cdots v_{k},w=w_1 w_2\cdots w_k\in\{0,1\}^{k}$, define $d(v,w)=\sum\limits_{i=1}^{k}|v_i-w_i|$.
We show next that $Q_v \cap Q_w \neq \emptyset$ whenever $d(v,w)=1$, which implies that $\bigcup\limits_{v\in\{0,1\}^{k}}Q_{v}$ is connected.
If $d(v,w)=1$, then there is an index $j\in\{1,2,\cdots,k\}$ such that $v_j\ne w_j$ and $v_i=w_i$ for $i\neq j$.
Without loss of generality, we assume that $v_j=0,w_j=1$.
It follows that
\begin{equation*}
  I_{u_j,0}=[u_j, u_j+r^{n+1}],\ I_{u_j,1}=[u_j+(1-r)r^{n}, u_j+r^{n}],\ I_{u_i,v_i}=I_{u_i,w_i}=[a_{i}, a_i+r^{n+1}],\ i\neq j,
\end{equation*}
where $a_i=u_i+v_i(r^n-r^{n+1})\geq u_i$ for $i\neq j$.
Then
\begin{align*}
& Q_v=\left[u_j^m+\sum\limits_{i\neq j}a_i^{m}, (u_j+r^{n+1})^{m}+\sum\limits_{i\neq j}(a_i+r^{n+1})^{m}\right],\\
& Q_w=\left[(u_j+(1-r)r^{n})^{m}+\sum\limits_{i\neq j}a_i^{m}, (u_j+r^{n})^{m}+\sum\limits_{i\neq j}(a_i+r^{n+1})^{m}\right].
\end{align*}
If we can prove
\begin{equation}\label{lem3.1}
  (u_j+r^{n+1})^{m}+\sum\limits_{i\neq j}(a_i+r^{n+1})^{m}\ge(u_j+(1-r)r^{n})^{m}+\sum\limits_{i\neq j}a_i^{m},
\end{equation}
then $Q_v \cap Q_w \neq \emptyset$.
Let $h(x)=-(u_j+r^{n+1}+\lambda x)^{m}+\sum\limits_{i\neq j}(a_i+x)^{m}$ for $x\in[0,r^{n+1}]$.
Then
\begin{equation*}
  h'(x)=m\left(-\lambda (u_j+r^{n+1}+\lambda x)^{m-1}+\sum\limits_{i\neq j}(a_i+x)^{m-1}\right).
\end{equation*}
Since $a_i+x \geq u_i$ for $i\neq j$ and  $u_j+r^{n+1}+\lambda x\leq u_M+r^{n+1}+\lambda r^{n+1}=u_M+r^n-r^{n+1}$ for $x\in[0,r^{n+1}]$, we have $\sum\limits_{i\neq j}\left(\frac{a_i+x}{u_j+r^{n+1}+\lambda x}\right)^{m-1}\geq\sum\limits_{i\neq j}\left(\frac{u_i}{u_M+r^n-r^{n+1}}\right)^{m-1}\geq \sum\limits_{i\neq M}\left(\frac{u_i}{u_M+r^n-r^{n+1}}\right)^{m-1}\geq \lambda$ from the condition $\sum\limits_{i\neq M}u_i^{m-1} \geq \lambda(u_M+r^n-r^{n+1})^{m-1}$.
This implies $h'(x)\geq 0$ for $x\in[0,r^{n+1}]$.
Thus, $h(r^{n+1})\ge h(0)$, which is equivalent to (\ref{lem3.1}).
The proof is complete.
\end{proof}

\begin{lemma}\label{lamconditionp}
Suppose that $k\in\mathbb{N}$, $k\geq 2$ and $u_{1},u_{2},\cdots,u_{k}\in L_{n}$.
If
\begin{equation*}
  \sum\limits_{i\neq M}u_i^{m-1}\geq \lambda(u_M+r^{n})^{m-1},
\end{equation*}
where $\lambda=\frac{1}{r}-2$ and $u_M=\max\{u_{1},u_{2},\cdots,u_{k}\}$, then $f_{k,m}(I_{u_{1}}\times I_{u_{2}}\times\cdots\times I_{u_{k}})\subseteq f_{k,m}(C_\alpha^k)$.
More precisely, we have
\begin{equation*}
  f_{k,m}(I_{u_{1}}\times I_{u_{2}}\times\cdots\times I_{u_{k}})=f_{k,m}((I_{u_{1}}\times I_{u_{2}}\times\cdots\times I_{u_{k}})\cap C_{\alpha}^k).
\end{equation*}
\end{lemma}

\begin{proof}
For any $i=1,2,\cdots,k$ and integer $l\geq n$, let $\mathcal{F}_{i,l}=\{I\in\mathcal{F}_l:I\subseteq I_{u_i}\}$ and  $F_{i,l}=\bigcup\limits_{A\in\mathcal{F}_{i,l}}A\subseteq I_{u_i}$.
Then $f_{k,m}(I_{u_{1}}\times I_{u_{2}}\times\cdots\times I_{u_{k}})\supseteq f_{k,m}(F_{1,l}\times F_{2,l}\times\cdots\times F_{k,l})$ for all $l\geq n$.
It follows from Lemma \ref{fkmCk} that $f_{k,m}(C_\alpha^{k})=\bigcap\limits_{l=1}^{\infty} f_{k,m}(F_l^{k})=\bigcap\limits_{l=n}^{\infty} f_{k,m}(F_l^{k})$.
Then by $F_{i,l}\subseteq F_{l}$ for $i=1,2,\cdots,k$, we have
\begin{equation*}
  \bigcap_{l=n}^{\infty}f_{k,m}(F_{1,l}\times F_{2,l}\times\cdots\times F_{k,l}) \subseteq \bigcap_{l=n}^{\infty} f_{k,m}(F_l^{k})=f_{k,m}(C_\alpha^{k}).
\end{equation*}
So it suffices to show that
\begin{equation}\label{lem3.2}
  f_{k,m}(I_{u_{1}}\times I_{u_{2}}\times\cdots\times I_{u_{k}})\subseteq f_{k,m}(F_{1,l}\times F_{2,l}\times\cdots\times F_{k,l})
\end{equation}
for all $l\ge n$.
We now prove it by induction on $l$.
If $l=n$, then $F_{i,n}=I_{u_i}$ for $i=1,2,\cdots,k$.
Thus, (\ref{lem3.2}) holds trivially.
We assume (\ref{lem3.2}) holds for some $l\geq n$, that is,
\begin{equation*}
f_{k,m}(I_{u_{1}}\times I_{u_{2}}\times\cdots\times I_{u_{k}})\subseteq f_{k,m}(F_{1,l}\times F_{2,l}\times\cdots\times F_{k,l}).
\end{equation*}
Take $y\in f_{k,m}(I_{u_{1}}\times I_{u_{2}}\times\cdots\times I_{u_{k}})$, then there are $u_1',u_2',\cdots,u_k'\in L_l$ such that
\begin{equation*}
  I_{u_i'}\in\mathcal{F}_{i,l}\ \text{and}\ y\in f_{k,m}(I_{u_1'}\times I_{u_2'}\times\cdots\times I_{u_k'}).
\end{equation*}
Let $u_{M'}'=\max\{u_{1}',u_{2}',\cdots,u_{k}'\}$.
Then $u_{M'}'\leq u_M+r^n-r^l$ and $u_i' \geq u_i$ for $i=1,2,\cdots, k$.
Combining this with the condition of the lemma, we derive
\begin{equation*}
  \sum\limits_{i\neq M'}(u_i')^{m-1}\geq \sum\limits_{i\neq M}u_i^{m-1}\geq \lambda(u_M+r^{n})^{m-1}\geq \lambda(u_{M'}'+r^{l})^{m-1}\ge\lambda(u_{M'}'+r^l-r^{l+1})^{m-1}.
\end{equation*}
By Lemma 3.1, we obtain that
\begin{equation*}
f_{k,m}(I_{u_{1}'}\times I_{u_{2}'}\times\cdots\times I_{u_{k}'})=f_{k,m}\left((I_{u_1',0}\cup I_{u_1',1})\times (I_{u_2',0}\cup I_{u_2',1})\times\cdots\times (I_{u_k',0}\cup I_{u_k',1})\right).
\end{equation*}
Hence, there is a $v=(v_1,v_2,\cdots,v_k)\in\{0,1\}^k$ such that $y\in f_{k,m}(I_{u_1',v_1}\times I_{u_2',v_2}\times\cdots\times I_{u_k',v_k})$.
Note that $I_{u_i',v_i}\in\mathcal{F}_{i,l+1}$ for $i=1,2,\cdots,k$.
We derive $y\in f_{k,m}(F_{1,l+1}\times F_{2,l+1}\times\cdots\times F_{k,l+1})$ and (\ref{lem3.2}) holds for $l+1$.
The proof is complete.
\end{proof}

The following lemma gives a lower bound of $G_\alpha(m)$.

\begin{lemma}\label{thmlowerbound}
Suppose that $\alpha>1$ and $r=\frac{1}{2}(1-\frac{1}{\alpha})$.
For $m\ge 1$, if positive integer $k<(\frac{1}{r}-1)^m$, then
\begin{equation*}
  [0,k]\ne\{x_1^m+x_2^m+\cdots+x_k^m: x_j\in C_{\alpha},1\le j\le k\}.
\end{equation*}
In other words, we have $G_\alpha(m)\geq (\frac{1}{r}-1)^m$.
\end{lemma}

\begin{proof}
Note that for any $x\in C_\alpha$, we have $x\leq r$ or $x\geq 1-r$.
Let $(x_1,x_2,\cdots,x_k)\in C_\alpha^k$.
\begin{enumerate}
  \item If $x_j\leq r$ for all $1\le j\le k$, then $f(x_1,x_2,\cdots,x_k)=\sum\limits_{j=1}^{k}x_{j}^m\leq kr^m<(1-r)^m$.
  \item If $x_{j_0}\geq 1-r$ for some $1\le j_0\le k$, then $f(x_1,x_2,\cdots,x_k)=\sum\limits_{j=1}^{k}x_{j}^m\ge x_{j_0}^m\geq (1-r)^m$.
\end{enumerate}
Thus, $(kr^m,(1-r)^m)$ is not contained in $f_{k,m}(C_\alpha^k)$ for $k<(\frac{1}{r}-1)^m$.
Consequently, we obtain the conclusion.
\end{proof}

For the rest of this section, we shall provide an upper bound of $G_\alpha(m)$ for $m\ge 1$ and show that $G_\alpha(m)=\lceil(\frac{1}{r}-1)^m\rceil$ under some technical assumptions.

\begin{theorem}\label{m=1}
Suppose that $\alpha>1$ and  $r=\frac{1}{2}(1-\frac{1}{\alpha})\in (0,\frac{1}{2})$.
Then $\lceil\frac{1}{r}-1\rceil$ is the smallest positive integer $k$ that satisfies
\begin{equation*}
  [0,k]=\{x_1+x_2+\cdots+x_k:x_j\in C_{\alpha},1\le j\le k\}.
\end{equation*}
In other words, we have $G_{\alpha}(1)=\lceil\frac{1}{r}-1\rceil$.
\end{theorem}

\begin{proof}
Let $k=\lceil\frac{1}{r}-1\rceil\ge 2$.
Consider Lemma \ref{lamconditionp} with $u_1=u_2=\cdots=u_k=1-r\in L_1$, we derive
\begin{equation*}
  [k(1-r),k]=\{x_1+x_2+\cdots+x_k:x_j\in[1-r,1]\cap C_{\alpha},1\le j\le k\}.
\end{equation*}
Note that $k-k(1-r)=kr\ge 1-r$.
Then for any $x\in [0,k]$, there is an integer $0\le l\le k$ such that $x+l(1-r)\in [k(1-r),k]$.
Assume $x+l(1-r)=x_1+x_2+\cdots+x_k$, where $x_j\in[1-r,1]\cap C_{\alpha}$ for $1\le j\le k$.
Let
\begin{equation*}
  y_j=\begin{cases}
        x_j-(1-r), & 1\le j\le l, \\
        x_j,& l+1\le j\le k.
      \end{cases}
\end{equation*}
It follows that $y_j\in C_{\alpha}$ for $1\le j\le k$ and $x=y_1+y_2+\cdots+y_k\in f_{k,1}(C_{\alpha}^k)$.
Furthermore, it follows from Lemma \ref{thmlowerbound} that such $k$ is the smallest.
Then the proof is complete.
\end{proof}

We now introduce some notation that will be used frequently in the following.
Define
\begin{align*}
  n_* & =n_*(r,m)=\lfloor-\log_r m\rfloor+1, \\
  k_* & =k_*(r,m)=\left\lfloor\lambda\left(1+\frac{r^{n_*}}{1-r}\right)^{m-1}\right\rfloor+2,\\
  l_0 & =l_0(r,m,k)=\left\lfloor m+\log_r\left(\frac{1}{\lambda}(1-r)(k-a)\right)\right\rfloor+1, \\
  m_0 & =m_0(r,m,k)=\left\lfloor m+\log_r\left(\frac{1}{\lambda}(1-r)(k+k_*-1-b)\right)\right\rfloor+1,
\end{align*}
where $\lambda=\frac{1}{r}-2$, $a=k_*(1-r)^m$ and $b=ar^m+(1-r)^m$.
For simplicity to presentation, we list some conditions on $k$.\\
$(A1)\quad k\ge \max\left\{k_*,\frac{\lambda}{(1-r)^{m-1}}+1\right\}$,\\
$(A2)\quad k\le \frac{\lambda}{(1-r)r^m}+a$,\\
$(A2') \quad k\le \frac{\lambda}{(1-r)r^m}+b+1-k_*$,\\
$(A3)\quad (k-a)r^m+(1-r+r^{l_0})^m\ge 2(1-r)^m$,\\
$(A4)\quad (k+k_*-1-b)\left(1+\frac{mr}{\lambda}(1-r)^m\right)\ge(\frac{1}{r}-1)^m+a-b$.\\
It is easy to see that $(A2')$ implies $(A2)$.

Applying Lemma \ref{lamconditionp}, we can find an interval in $f_{k,m}(C_\alpha^k)$, which is helpful to obtain an upper bound of $G_\alpha(m)$.

\begin{lemma}\label{lemk1}
If integer $k$ satisfies $(A1)$, then $[k_*(1-r)^m,k]\subseteq f_{k,m}(C_\alpha^k)$.
\end{lemma}

\begin{proof}
Using the fact that $k_*>\lambda\left(1+\frac{r^{n_*}}{1-r}\right)^{m-1}+1$, we obtain $(k_*-1)(1-r)^{m-1}\ge\lambda(1-r+r^{n_*})^{m-1}$.
It follows from Lemma \ref{lamconditionp} that $[k_*(1-r)^m,k_*(1-r+r^{n_*})^m]\subseteq f_{k_*,m}(C_\alpha^{k_*})$.
Note that $n_*\leq -\log_r m+1$.
We have $r^{n_*}\ge\frac{r}{m}$.
Then
\begin{align*}
  k_*\left((1-r+r^{n_*})^m-(1-r)^m\right) & \ge(\lambda+1)(1-r)^m\left(\left(1+\frac{r^{n_*}}{1-r}\right)^m-1\right)\\
  & \ge(\lambda+1)(1-r)^m\frac{m r^{n_*}}{1-r} \\
  & \ge(\lambda+1)(1-r)^m\frac{r}{1-r}=(1-r)^m.
\end{align*}
Since $0,1-r\in C_\alpha$ and $k\ge k_*$, we obtain
\begin{equation}\label{lem3.51}
  [k_*(1-r)^m,k_*(1-r+r^{n_*})^m+(k-k_*)(1-r)^m]\subseteq f_{k,m}(C_\alpha^{k})
\end{equation}
Notice that $k\ge \frac{\lambda}{(1-r)^{m-1}}+1$.
Then $(k-1)(1-r)^{m-1}\ge \lambda$.
Again, by Lemma \ref{lamconditionp}, we have
\begin{equation}\label{lem3.52}
  [k(1-r)^m,k]\subseteq f_{k,m}(C_\alpha^k).
\end{equation}
Combining (\ref{lem3.51}) and (\ref{lem3.52}), we see that $[k_*(1-r)^m,k]\subseteq f_{k,m}(C_\alpha^k)$ for $k\ge\max\left\{k_*,\frac{\lambda}{(1-r)^{m-1}}+1\right\}$.
\end{proof}

Using the above result, we can give an upper bound of $G_\alpha(m)$.

\begin{theorem}\label{lemupperbound}
Suppose that $m\ge 1$, $\alpha>1$ and $r=\frac{1}{2}(1-\frac{1}{\alpha})$.
Let
\begin{equation*}
  M=M(r,m)=\max\left\{\left(\frac{1}{r}-1\right)^m+\left(\lambda e^{\frac{1}{1-r}}+2\right)(1-r)^m,\lambda e^{\frac{1}{1-r}}+2,\frac{\lambda}{(1-r)^{m-1}}+1\right\},
\end{equation*}
where $\lambda=\frac{1}{r}-2$.
Then for each integer $\kappa\geq M+\lambda e^{\frac{1}{1-r}}+2$, we have
\begin{equation*}
  [0,\kappa]=\{x_1^m+x_2^m+\cdots+x_{\kappa}^m:x_j\in C_\alpha,1\le j\le \kappa\}.
\end{equation*}
In other words, we have $G_{\alpha}(m)\le M+\lambda e^{\frac{1}{1-r}}+3$.
In particular, $G_{\alpha}(m)$ is finite.
\end{theorem}

\begin{proof}
We have $\left(1+\frac{r^{n_*}} {1-r}\right)^{m-1}\le\left(1+\frac{1}{(1-r)m}\right)^{m-1}\le e^{\frac{1}{1-r}}$ since $n_*\ge -\log_r m$.
So $k_*\le\left\lfloor\lambda e^{\frac{1}{1-r}}\right\rfloor+2$.
It follows that $M\ge\max\left\{\left(\frac{1}{r}-1\right)^m+k_*(1-r)^m,k_*,\frac{\lambda}{(1-r)^{m-1}}+1\right\}$.
Suppose $k\ge M$.
From Lemma \ref{lemrk} and Lemma \ref{lemk1}, we obtain $[k_*(1-r)^mr^m,kr^m]\subseteq f_{k,m}(C_{\alpha}^k)$
and $r^m\left(k-k_*(1-r)^m\right)\ge (1-r)^m$.

Let $\kappa=k+\left\lfloor\lambda e^{\frac{1}{1-r}}\right\rfloor+2\ge k+k_*$.
Since $0,1-r\in C_\alpha$, we derive
\begin{equation*}
  [k_*(1-r)^mr^m,kr^m+k_*(1-r)^m]\subseteq f_{\kappa,m}(C_{\alpha}^{\kappa}).
\end{equation*}
Invoking Lemma \ref{lemk1} again, we obtain $[k_* (1-r)^m,\kappa]\subseteq f_{\kappa,m}(C_{\alpha}^{\kappa})$.
Then we can see that
\begin{align*}
  [\kappa r^m,\kappa]&\subseteq[k_*(1-r)^mr^m,\kappa]\\
  &=[k_*(1-r)^mr^m,kr^m+k_*(1-r)^m]\cup[k_* (1-r)^m,\kappa]\subseteq f_{\kappa,m}(C_{\alpha}^{\kappa}).
\end{align*}
So $f_{\kappa,m}(C_\alpha^{\kappa})\supseteq\bigcup\limits_{l\ge 0}[\kappa r^{(l+1)m},\kappa r^{lm}]=(0,\kappa]$ from Lemma \ref{lemrk}.
Hence, we obtain $f_{\kappa,m}(C_\alpha^{\kappa})=[0,\kappa]$ since $0\in C_\alpha$.
This completes the proof.
\end{proof}

For the case $0<r<\frac{3-\sqrt{5}}{2}<\frac{1}{2}$ (that is, $\frac{1}{r}-1>\frac{1}{1-r}$), we can determine $G_\alpha(m)$ if $m$ is large enough.
We first give several lemmas.
The next lemma gives an upper bound of $\text{Gap}(f_{1,m}([1-r,1-r+r^l]\cap C_\alpha))$.

\begin{lemma}\label{lemgapf1m}
For $l\ge 1$, we have $\text{Gap}(f_{1,m}([1-r,1-r+r^l]\cap C_\alpha))\le \frac{\lambda r^l}{1-r}$, where $\lambda=\frac{1}{r}-2$.
\end{lemma}

\begin{proof}
If $(a,b)\subseteq [1-r,1-r+r^l]\backslash C_\alpha$ with $a,b\in C_{\alpha}$, then we have $b-a=(1-2r)r^n$ and $b\le 1-r+r^l-r^{n+1}$ for some $n\ge l$.
So $a\le 1-r+r^l-r^{n+1}-(1-2r)r^n=1-r+r^l-(1-r)r^n$.
Let $t=\frac{1-a}{1-b}>1$ and $x=1-b\in(0,\frac{1}{t})$.
It follows from Lemma \ref{inequality} that
\begin{align*}
  b^m-a^m&=(1-x)^m-(1-tx)^m\le 1-\frac{1}{t}=1-\frac{1-b}{1-a}=\frac{b-a}{1-a}\\
  &\le\frac{(1-2r)r^n}{r-r^l+(1-r)r^n}\le \frac{(1-2r)r^l}{r-r^l+(1-r)r^l}=\frac{(1-2r)r^l}{r-r^{l+1}}
  =\frac{\lambda r^l}{1-r^l}\le \frac{\lambda r^l}{1-r}.
\end{align*}
This completes the proof.
\end{proof}

To sharp the upper bound of $G_\alpha(m)$ in Theorem \ref{lemupperbound}, we need to optimize Lemma \ref{lemk1} by the following two lemmas.

\begin{lemma}\label{lemk2}
Suppose that integer $k$ satisfies $(A1),(A2)$ and $(A3)$.
If integer $\kappa\ge k+k_*-1$, then we have
\begin{equation*}
[k_*(1-r)^mr^m+(1-r)^m,\kappa]\subseteq f_{\kappa,m}(C_\alpha^{\kappa}).
\end{equation*}
\end{lemma}

\begin{proof}
Using Lemma \ref{lemk1}, we obtain $[a,k]\subseteq f_{k,m}(C_\alpha^k)$, where $a=k_*(1-r)^m$.
Since $k$ satisfies $(A2)$, we derive $m+\log_r\left(\frac{1}{\lambda}(1-r)(k-a)\right)\ge 0$, then $l_0=\left\lfloor m+\log_r\left(\frac{1}{\lambda}(1-r)(k-a)\right)\right\rfloor+1\ge 1$.
Combining this with Lemma \ref{lemgapf1m}, we can see that
$\text{Gap}(f_{1,m}([1-r,1-r+r^{l_0}]\cap C_\alpha))\le \frac{\lambda r^{l_0}}{1-r}\le (k-a)r^m$.
From Lemmas \ref{lemgap} and \ref{lemrk}, we obtain
\begin{align*}
  f_{k+1,m}(C_\alpha^{k+1})&\supseteq f_{1,m}([1-r,1-r+r^{l_0}]\cap C_\alpha)+[ar^m,k r^m]\\
   &=[ar^m+(1-r)^m,k r^m+(1-r+r^{l_0})^m].
\end{align*}
Since $k$ satisfies (A3), we have $kr^m+(1-r+r^{l_0})^m\ge 2(1-r)^m+ar^m$.
It follows that
\begin{equation*}
  [ar^m+(1-r)^m,ar^m+2(1-r)^m]\subseteq f_{k+1,m}(C_\alpha^{k+1}).
\end{equation*}
Note that $0,1-r\in C_{\alpha}$.
We can see
\begin{equation}\label{lem3.81}
  [ar^m+(1-r)^m,ar^m+k_*(1-r)^m]\subseteq f_{k+k_*-1,m}(C_\alpha^{k+k_*-1})\subseteq f_{\kappa,m}(C_\alpha^{\kappa}).
\end{equation}
Again, by Lemma \ref{lemk1}, we have
\begin{equation}\label{lem3.82}
  [k_*(1-r)^m,\kappa]\subseteq f_{\kappa,m}(C_\alpha^{\kappa}).
\end{equation}
Combining (\ref{lem3.81}) and (\ref{lem3.82}), we obtain $[k_*(1-r)^mr^m+(1-r)^m,\kappa]\subseteq f_{\kappa,m}(C_\alpha^{\kappa})$.
This completes the proof.
\end{proof}

\begin{lemma}\label{lemk3}
If integer $k$ satisfies $(A1),(A2'),(A3)$ and $(A4)$, then $[(1-r)^m,\kappa]\subseteq f_{\kappa,m}(C_\alpha^{\kappa})$, where $\kappa\in\mathbb{N}$ and $\kappa\ge k+k_*$.
\end{lemma}

\begin{proof}
From Lemma \ref{lemk2}, we obtain that $[b,k']\subseteq f_{k',m}(C_\alpha^{k'})$, where $k'=k+k_*-1$, $a=k_*(1-r)^m$ and $b=ar^m+(1-r)^m$.
Since $k$ satisfies $(A2')$, we see that $m+\log_r\left(\frac{1}{\lambda}(1-r)(k'-b)\right)\ge 0$.
Then we derive $m_0=\left\lfloor m+\log_r\left(\frac{1}{\lambda}(1-r)(k'-b)\right)\right\rfloor+1\ge 1$ and $m_0\ge m+\log_r\left(\frac{1}{\lambda}(1-r)(k'-b)\right)$.
It follows from Lemma \ref{lemgapf1m} that
\begin{equation*}
  \text{Gap}(f_{1,m}([1-r,1-r+r^{mn+m_0}]\cap C_\alpha))\le \frac{\lambda r^{mn+m_0}}{1-r}\le (k'-b)r^{mn+m},\ \forall n\ge 0.
\end{equation*}
By Lemma \ref{lemrk}, we have $[br^{m(n+1)},k'r^{m(n+1)}]\subseteq f_{k',m}(C_\alpha^{k'})$ for $n\geq 0$.
Combining this with Lemma \ref{lemgap}, we can see that
\begin{align*}
  f_{k'+1,m}(C_\alpha^{k'+1})&\supseteq f_{1,m}([1-r,1-r+r^{mn+m_0}]\cap C_\alpha)+[b r^{m(n+1)},k'r^{m(n+1)}]\\
  & =[b r^{m(n+1)}+(1-r)^m,k'r^{m(n+1)}+(1-r+r^{mn+m_0})^m].
\end{align*}
Let
\begin{equation*}
  T_n=[b r^{m(n+1)}+(1-r)^m,k'r^{m(n+1)}+(1-r+r^{mn+m_0})^m],\ \forall n\ge 0.
\end{equation*}
Since $k$ satisfies $(A4)$ and $m_0\leq m+\log_r\left(\frac{1}{\lambda}(1-r)(k'-b)\right)+1$, we have
\begin{align*}
  (1-r+r^{mn+m_0})^m-(1-r)^m & \geq m(1-r)^{m-1}r^{mn+m_0}\ge \frac{mr}{\lambda}(1-r)^m(k'-b)r^{m(n+1)} \\
  & \ge br^{mn}-k'r^{m(n+1)}.
\end{align*}
That is, $k'r^{m(n+1)}+(1-r+r^{mn+m_0})^m\ge b r^{mn}+(1-r)^m$.
Then $T_n$ is connected with $T_{n-1}$ for $n\ge 1$.
Hence,
\begin{equation*}
  ((1-r)^m,k'r^m+(1-r+r^{m_0})^m]=\bigcup\limits_{n=0}^{\infty}T_n\subseteq f_{k'+1,m}(C_\alpha^{k'+1}).
\end{equation*}
Again, from Lemma \ref{lemk2}, we obtain that $[b,\kappa]\subseteq f_{\kappa,m}(C_\alpha^{\kappa})$.
Note that
\begin{equation*}
  b=ar^m+(1-r)^m<k'r^{m}+(1-r+r^{m_0})^m.
\end{equation*}
We derive $((1-r)^m,\kappa]\subseteq f_{\kappa,m}(C_\alpha^{\kappa})$.
Therefore, we obtain the conclusion since $(1-r)^m\in f_{\kappa,m}(C_\alpha^{\kappa})$.
\end{proof}

According to the above lemmas, we can determine $G_{\alpha}(m)$ for $1<\alpha<2+\sqrt{5}$ (that is, $\frac{1}{r}-1>\frac{1}{1-r}$).

\begin{theorem}\label{upperbound}
Suppose that $1<\alpha<2+\sqrt{5}$
and $m\ge\frac{\log\left((2-r)(r^{-1}\lambda+1-r)\left(\lambda e^{\frac{1}{1-r}}+2\right)\right)}{\log\left(\left(\frac{1}{r}-1\right)(1-r)\right)}$, where $r=\frac{1}{2}(1-\frac{1}{\alpha})$ and $\lambda=\frac{1}{r}-2$.
The smallest positive integer $\kappa$ that satisfies
\begin{equation*}
  [0,\kappa]=\{x_1^m+x_2^m+\cdots+x_{\kappa}^m: x_j\in C_{\alpha}, 1\le j\le \kappa\}
\end{equation*}
is $\lceil(\frac{1}{r}-1)^m\rceil$.
In other words, we have $G_{\alpha}(m)=\lceil(\frac{1}{r}-1)^m\rceil$ for all sufficiently large $m$.
\end{theorem}

\begin{proof}
Let $\kappa=\lceil(\frac{1}{r}-1)^m\rceil\ge(\frac{1}{r}-1)^m$.
If we can prove $[(1-r)^m,\kappa]\subseteq f_{\kappa,m}(C_\alpha^{\kappa})$, then combining Lemma \ref{lemrk} and $\kappa r^m\ge (1-r)^m$, we derive $[0,\kappa]=f_{\kappa,m}(C_{\alpha}^{\kappa})$.
Now we show that $[(1-r)^m,\kappa]\subseteq f_{\kappa,m}(C_\alpha^{\kappa})$.
By Lemma \ref{lemk3}, we only need to check that $k=\kappa-k_*$ satisfies $(A1),(A2'),(A3)$ and $(A4)$.

We have $0<r<\frac{3-\sqrt{5}}{2}$ and $(2-r)(r^{-1}\lambda+1-r)\ge 2$ for $1<\alpha<2+\sqrt{5}$.
Then $m\ge\frac{\log2\left(\lambda e^{\frac{1}{1-r}}+2\right)}{\log\left(\frac{1}{r}-1\right)}$.
Since $(2-r)(r^{-1}\lambda+1-r)(\lambda e^{\frac{1}{1-r}}+2)\ge 2(\lambda e^{\frac{1}{1-r}}+2)\ge\lambda+\lambda e^{\frac{1}{1-r}}+3$, we have
$m\ge\frac{\log\left(\left(\lambda+\lambda e^{\frac{1}{1-r}}+3\right)(1-r)\right)}{\log\left(\left(\frac{1}{r}-1\right)(1-r)\right)}$.
As in the proof of Theorem \ref{lemupperbound},
we obtain $k_*\le\lambda e^{\frac{1}{1-r}}+2$.
Taking $m\ge\frac{\log2\left(\lambda e^{\frac{1}{1-r}}+2\right)}{\log\left(\frac{1}{r}-1\right)}$ into consideration, we derive $(\frac{1}{r}-1)^m\ge 2k_*$.
So $k\ge(\frac{1}{r}-1)^m-k_*\ge k_*$.
Since $m\ge\frac{\log\left(\left(\lambda+\lambda e^{\frac{1}{1-r}}+3\right)(1-r)\right)}{\log\left(\left(\frac{1}{r}-1\right)(1-r)\right)}$, we can see
\begin{equation*}
  \left(\left(\frac{1}{r}-1\right)(1-r)\right)^m\ge(\lambda+k_*+1)(1-r)\ge\lambda(1-r)+(k_*+1)(1-r)^m.
\end{equation*}
It follows that $(\frac{1}{r}-1)^m\ge\frac{\lambda}{(1-r)^{m-1}}+k_*+1$ and  $k\ge(\frac{1}{r}-1)^m-k_*\ge\frac{\lambda}{(1-r)^{m-1}}+1$.
Thus, $(A1)$ is valid.

We have $(1-r)^2<\frac{1}{r}-2$ for $0<r<\frac{3-\sqrt{5}}{2}$, then
\begin{equation*}
  (1-r)^m\le 1-r\le\frac{\lambda}{1-r}\le\frac{\lambda}{1-r}+(1-r)^mr^m+ar^{2m},
\end{equation*}
where $a=k_*(1-r)^m$.
It follows that $(\frac{1}{r}-1)^m\le\frac{\lambda}{(1-r)r^m}+(1-r)^m+ar^m$.
Thus, $(A2')$ is valid since $k\le (\frac{1}{r}-1)^m-k_*+1$.

Since $l_0\le m+\log_r\left(\frac{1}{\lambda}(1-r)(k-a)\right)+1$, we have
\begin{equation*}
  (1-r+r^{l_0})^m-(1-r)^m\ge m(1-r)^{m-1}r^{l_0}\ge\frac{mr}{\lambda}(1-r)^m(k-a)r^m.
\end{equation*}
To prove $(A3)$, it suffices to show that $\frac{mr}{\lambda}(1-r)^m(k-a)\ge k_*+a$ since $kr^m\ge (1-r)^m-k_*r^m$.
Notice that $m\ge\frac{\log\left((2-r)(r^{-1}\lambda+1-r)\left(\lambda e^{\frac{1}{1-r}}+2\right)\right)}{\log\left(\left(\frac{1}{r}-1\right)(1-r)\right)}$.
We obtain
\begin{align*}
  \left(\left(\frac{1}{r}-1\right)(1-r)\right)^m&\ge (2-r)(r^{-1}\lambda+1-r)\left(\lambda e^{\frac{1}{1-r}}+2\right)\\
  &\ge (1+(1-r)^m)\left(\frac{\lambda}{mr}+(1-r)^m\right)k_*\\
  &=(k_*+a)\left(\frac{\lambda}{mr}+(1-r)^m\right).
\end{align*}
That is, $\frac{mr}{\lambda}(1-r)^m\left(\left(\frac{1}{r}-1\right)^m-k_*-a\right)\ge k_*+a$.
Note that $k\ge (\frac{1}{r}-1)^m-k_*$.
Thus, $(A3)$ is valid.

Moreover, we can see that
\begin{equation*}
  \frac{mr}{\lambda}(1-r)^m\left(\frac{1}{r}-1\right)^m\ge k_*+a+\frac{mr}{\lambda}(k_*+a)(1-r)^m\ge 1+a+\frac{mr}{\lambda}(1+b)(1-r)^m.
\end{equation*}
This implies $(A4)$.
The proof is complete.
\end{proof}

We end this section by the following conjecture.

\begin{conjecture}
Suppose that $\alpha>1$ and $r=\frac{1}{2}(1-\frac{1}{\alpha})$.
For any real number $m\ge 1$ and integer $k=\left\lceil\left(\frac{1}{r}-1\right)^m\right\rceil$,
\begin{equation*}
  \left[0,k\right]
  =\left\{x_1^m+x_2^m+\cdots+x_{k}^m: x_j\in C_{\alpha}, 1\le j\le k\right\}.
\end{equation*}
\end{conjecture}

\section{Sum of $m$-th powers of elements in $C$ and $C_2$}
With the notation of $k_*$, $n_*$, $l_0$ and conditions $(A1),(A2'),(A3),(A4)$ on $k$ in Section 3, we investigate two conjectures of Guo in this section.

For the Cantor ternary set, Guo conjectured that $G(m)=2^m$ for any integer $m\geq 4$.
We now prove this conjecture.

\begin{theorem}\label{2^m}
  For integer $m\ge 4$, we have
  \begin{equation*}
    [0,2^m]=\{x_1^m+x_2^m+\cdots+x_{2^m}^m:x_j\in C,1\le j\le 2^m\}.
  \end{equation*}
\end{theorem}

\begin{proof}
Let $r=\frac{1}{3}$, $\lambda=1$ and $k=2^m-k_*$.
From Lemma \ref{lemk3}, we only need to check that $k$ satisfies $(A1),(A2'),(A3)$ and $(A4)$.
Similar to the proof of Theorem \ref{upperbound}, $(A2')$ holds for $0<r=\frac{1}{3}<\frac{3-\sqrt{5}}{2}$.

Since $k_*\le e^{\frac{1}{1-r}}+2=e^{\frac{3}{2}}+2<7$, we obtain $k_*\le 6$.
It follows that $k\ge 16-k_*>k_*$.
Moreover, we have $k\ge 2^m-6\ge (\frac{3}{2})^{m-1}+1$ for $m\ge 4$.
Thus, $(A1)$ is valid.

The condition $(A3)$ is equivalent to
\begin{equation}\label{thm4.1}
  (1-r+r^{l_0})^m-(1-r)^m\ge r^m(k_*+a),
\end{equation}
where $a=k_*(1-r)^m$.
\begin{enumerate}
  \item For $m=4,5$, we have $n_*=2$, $k_*=3$, $l_0=3$ and $a\le\frac{16}{27}$, then (\ref{thm4.1}) holds.
  \item For $m\ge 6$, we have $a\le 6(\frac{2}{3})^6<1$.
  Note that $l_0\le m+\log_r\left((1-r)(k-a)\right)+1$.
  We derive
\begin{align*}
  (1-r+r^{l_0})^m-(1-r)^m & \ge m(1-r)^{m-1}r^{l_0}\ge mr(1-r)^m(k-a)r^m \\
  & \ge\frac{m}{3}\left(\frac{2}{3}\right)^m(2^m-7)r^m\ge 7r^m\ge (k_*+a)r^m.
\end{align*}
\end{enumerate}
Thus, $(A3)$ is valid.

The condition $(A4)$ is equivalent to $mr(1-r)^m(2^m-1-b)\ge 1+a$.
For $m\ge 4$, we have $a\le 6(\frac{2}{3})^4<2$ and $b=(1-r)^m+ar^m\le(\frac{2}{3})^4+a(\frac{1}{3})^4<1$.
It follows that
\begin{equation*}
  mr(1-r)^m(2^m-1-b)\ge\frac{4}{3}\left(\frac{2}{3}\right)^m(2^m-2)\ge\frac{4}{3}\left(\frac{2}{3}\right)^4(2^4-2)>3>1+a.
\end{equation*}
Thus, $(A4)$ is valid.
\end{proof}

Using Theorem $\ref{upperbound}$, we can partly confirm Guo's conjecture on $C_2$.

\begin{theorem}
For integer $m\ge 7$, we have
\begin{equation*}
  [0,3^m]=\{x_1^m+x_2^m+\cdots+x_{3^m}^m:x_j\in C_2,j=1,2,\cdots,3^m\}.
\end{equation*}
\end{theorem}

\begin{proof}
Let $r=\frac{1}{4}$ and $\lambda=2$.
Then $\frac{\log\left((2-r)(r^{-1}\lambda+1-r)\left(\lambda e^{\frac{1}{1-r}}+2\right)\right)}{\log\left(\left(\frac{1}{r}-1\right)(1-r)\right)}\thickapprox 6.15233<7$.
Therefore, we obtain the conclusion by Theorem \ref{upperbound}.
\end{proof}

\section{Sums of $m$-th powers of two and three  elements in $C$}

We consider firstly sums of two $m$-th powers of elements in $C$ with real number $m>1$.

\begin{theorem}
  For any $\varepsilon\in(0,2)$ and any $m>1$, there is an interval
  \begin{equation*}
    I_{m,\varepsilon}\subseteq (2-\varepsilon,2)\backslash\{x_1^m+x_2^m: x_1,x_2\in C\}.
  \end{equation*}
  In particular, the set $\{x_1^m+x_2^m: x_1,x_2\in C\}$ cannot cover $(2-\varepsilon,2)$.
\end{theorem}

\begin{proof}
For $n\geq 2$, let $I_{u_1}, I_{u_2}$ and $I_{u_3}$ be the last three intervals in $\mathcal{F}_n$ and $u_1<u_2<u_3$.
Then we have $u_1=1-2r^{n-1}-r^n, u_2=1-r^{n-1}$ and $u_3=1-r^n$.
Moreover,
\begin{align*}
f_{2,m}(I_{u_{1}}\times I_{u_3})=[(1-2r^{n-1}-r^n)^m+(1-r^{n})^m,(1-2r^{n-1})^m+1]\subseteq f_{2,m}(F_{n}^{2}),\\
f_{2,m}(I_{u_2}\times I_{u_2})=[2(1-r^{n-1})^m,2(1-2r^n)^m]\subseteq f_{2,m}(F_{n}^{2}),\\
f_{2,m}(I_{u_2}\times I_{u_3})=[(1-r^{n-1})^{m}+(1-r^{n})^m,(1-2r^n)^m+1]\subseteq f_{2,m}(F_{n}^{2}).
\end{align*}
For $u,v\in L_n$ and $u\leq v$, let $L_{u,v}$ and $R_{u,v}$ be the left endpoint and the right endpoint of $f_{2,m}(I_u\times I_v)$ respectively.
Note that $(t^m)''=m(m-1)t^{m-2}>0$ for $t\in(0,\infty)$.
For $t_1,t_2\in(0,+\infty)$ and $t_1\neq t_2$, we derive $(t_1)^m+(t_2)^m>2(\frac{t_1+t_2}{2})^m$ from Jensen inequality.
Then we have
\begin{equation*}
  L_{u_2,u_3}=(1-r^{n-1})^m+(1-r^n)^m>2(1-2r^n)^m=R_{u_2,u_2},\ \forall n\ge 1.
\end{equation*}
Since $(1+t)^m=1+mt+O(t^2)$ as $t\rightarrow 0$, we have
\begin{align*} R_{u_1,u_3}-R_{u_2,u_2}&=1+(1-2r^{n-1})^m-2(1-2r^n)^m=1+(1-6r^n)^m-2(1-2r^n)^m\\
&=1+(1-6mr^n)-2(1-2mr^n)+O(r^{2n})=-2mr^n+O(r^{2n})
\end{align*}
as $n\rightarrow\infty$.
Then there is an $N>0$, such that for each $n>N$, $R_{u_1,u_3}<R_{u_2,u_2}<L_{u_2,u_3}$.
Note that for $n>N$, we have $R_{u,u_3}\leq R_{u_1,u_3}<L_{u_2,u_3}$ for $u\leq u_1$, $R_{u,v}\leq R_{u_2,u_2}<L_{u_2,u_3}$ for $v\leq u_2$ and $L_{u_3,u_3}>L_{u_2,u_3}$.
It follows that $(R_{u_2,u_2},L_{u_2,u_3})\cap f_{2,m}(I\times I')=\emptyset$ for any $n>N$ and any $I,I'\in\mathcal{F}_n$.
So $(R_{u_2,u_2},L_{u_2,u_3})\subseteq [0,2]\setminus\left(f_{2,m}(F_n^2)\right)$ for $n>N$.
By Lemma \ref{fkmCk}, we have $f_{2,m}(C^{2})=\bigcap\limits_{n=1}^{\infty} f_{2,m}(F_n^{2})$.
So
\begin{equation*}
  \left(2(1-2r^n)^m,(1-r^{n-1})^m+(1-r^n)^m\right)\subseteq [0,2]\backslash f_{2,m}(C^2),\  \forall n>N.
\end{equation*}
For any $\varepsilon\in(0,2)$, there is an integer $N'>N$, such that $2(1-2r^{N'})^m>2-\varepsilon$ since $\lim\limits_{n\rightarrow \infty}2(1-2r^n)^m=2$.
Let $I_{m,\varepsilon}=\left(2(1-2r^{N'})^m,(1-r^{N'-1})^m+(1-r^{N'})^m\right)$.
Then $I_{m,\varepsilon}\subseteq(2-\varepsilon,2)\backslash f_{2,m}(C^2)$.
Consequently, the set $\{x_1^m+x_2^m: x_1,x_2\in C\}$ cannot cover $(2-\varepsilon,2)$.
\end{proof}

We next consider sums of $m$-th powers of three elements in $C$ with $m\ge 1$.
Choosing appropriate $n$ in Lemma \ref{lamconditionp} for the case $k=3$, we can obtain the following result.

\begin{theorem}\label{thm33}
For any $m\geq 1$, there is an $\epsilon\ge\frac{1}{2}$, such that
\begin{equation*}
  [3-\epsilon,3]\subseteq \{x_1^m+x_2^m+x_3^m:x_1,x_2,x_3\in C\}.
\end{equation*}
\end{theorem}

\begin{proof}
Let $k=3$ and $r=\frac{1}{3}$.
Then $(k-1)(1-r^n)^{m-1}=2\left(1-(\frac{1}{3})^n\right)^{m-1}$ and $\lambda=\frac{1}{r}-2=1$.
If $n=\left\lfloor-\log_3\left(1-2^{-\frac{1}{m}}\right)\right\rfloor+1>-\log_3\left(1-2^{-\frac{1}{m}}\right)$,
then $2\left(1-(\frac{1}{3})^n\right)^{m-1}\ge 2\left(1-(\frac{1}{3})^n\right)^m\ge 1$.
Using Lemma \ref{lamconditionp}, we see that $[3-\epsilon,3]\subseteq f_{3,m}(C^3)$, where $\epsilon=3-3\left(1-(\frac{1}{3})^n\right)^m$.

Since $n\le-\log_3(1-2^{-\frac{1}{m}})+1$, we have $\left(1-(\frac{1}{3})^n\right)^m\le\left(1-\frac{1-2^{-\frac{1}{m}}}{3}\right)^m$.
Let $h(x)=\frac{1}{x}\log\left(1-\frac{1-e^{-x}}{3}\right)$ for $x>0$.
Then its derivative
\begin{equation*}
  h'(x)=-\frac{1}{x^2}\left(\log\left(1-\frac{1-e^{-x}}{3}\right)+\frac{x}{1+2e^x}\right),\ x>0.
\end{equation*}
Note that $\log(1+x)\leq x$ for $x>-1$.
We have $-\log(1-x)\geq x$ for $x\in(0,1)$.
So
\begin{equation*}
  h'(x)\ge\frac{1}{x^2}\left(\frac{1-e^{-x}}{3}-\frac{x}{1+2e^x}\right)
  =\frac{(1-e^{-x})(1+2e^x)-3x}{3x^2(1+2e^x)},\ x>0.
\end{equation*}
It is easy to see that $(1-e^{-x})(1+2e^x)=(e^x-e^{-x})+e^x-1\ge 2x+x=3x$ for $x>0$.
Then $h'(x)\ge 0$ and $h(x)$ is non-decreasing for $x>0$.
Thus, we derive $\left(1-\frac{1-2^{-\frac{1}{m}}}{3}\right)^m=2^{h\left(\frac{\log 2}{m}\right)}\leq 2^{h(\log 2)}=\frac{5}{6}$ for $m\ge 1$.
Therefore, we obtain $\epsilon=3-3\left(1-\left(\frac{1}{3}\right)^n\right)^m\geq 3-3\left(1-\frac{1-2^{-\frac{1}{m}}}{3}\right)^m \ge\frac{1}{2}$.
\end{proof}

\section{Cantor dust}

Let $r=\frac{1}{3}$ and $m\geq 3$ be an integer throughout this section.
Let $W=\{x+iy: x,y\in C\}=C+iC$ be the naturally embedded image of the Cantor dust $C\times C$ into the complex plane $\mathbb{C}$.
Hence, $W\cong C\times C$.
Let $S=\{x+iy:-1\leq x,y\leq 1\}=[-1,1]+i[-1,1]$ be a closed square centered at the origin.

Our main results in this section are stated as follows.

\begin{theorem}
  For each integer $m\geq 3$ and $k\geq 2^m$, we have
  \begin{equation*}
    \frac{k}{100}\left(\frac{2}{3}\right)^{m}S\subseteq\{z_1^m+z_2^m+\cdots+z_{4k}^m:
    z_j\in W,1\leq j\leq 4k\}.
  \end{equation*}
\end{theorem}

Using the above theorem, we obtain the following result.

\begin{theorem}
   For each integer $m\geq 3$, there is a positive integer $k\leq 2^{m+8}$, such that
  \begin{equation*}
    \{z\in\mathbb{C}: |z|\leq 1\}\subseteq\{z_1^m+z_2^m+\cdots+z_{k}^m:z_j\in W,1\leq j\leq k\}.
  \end{equation*}
\end{theorem}

\begin{proof}
Let $k=2^{m+8}$ and $k'=\frac{k}{4}=2^{m+6}$.
Since $m\ge 3$, we have $k'>2^m$ and $\frac{k'}{100}(\frac{2}{3})^{m}>1$.
It follows from Theorem 6.1 that $\{z\in\mathbb{C}: |z|\leq 1\}\subseteq S\subseteq\frac{k'}{100}(\frac{2}{3})^{m}S\subseteq f_{4k',m}(W^{4k'})=f_{k,m}(W^k)$.
\end{proof}

For convenience, we denote $\arctan r^n$ by $\theta_n$ for $n\in\mathbb{Z}$.
Then $\theta_{-n}=\arctan r^{-n}=\frac{\pi}{2}-\arctan r^n=\frac{\pi}{2}-\theta_n$.
From Lemma \ref{inequality} we can obtain the following inequality between $\theta_n$ and $\theta_{n+1}$ for $n\ge -1$.

\begin{lemma}\label{lem3beijiao}
  For integer $n\geq -1$, we have $\frac{1}{3}\theta_n\le\theta_{n+1}<\theta_n\leq 3\theta_{n+1}-r^{3n+3}$.
\end{lemma}

\begin{proof}
It is obvious that $\theta_{n+1}=\arctan(r^{n+1})<\arctan r^n=\theta_n$.
From Lemma \ref{inequality}, we see that $\arctan(3x) \leq 3\arctan x-x^3$ for $x\in[0,1]$, then
   \begin{equation*}
     \theta_n=\arctan (3r^{n+1}) \leq  3\arctan r^{n+1}-r^{3n+3}=3\theta_{n+1}-r^{3n+3},\  n\geq -1.
   \end{equation*}
This also implies $\frac{1}{3}\theta_n\le\theta_{n+1}$.
\end{proof}

Using the fact that $C\subseteq W$, we can find some segments in $f_{k,m}(W^k)$ from Theorem \ref{2^m}.

\begin{lemma}\label{lemK}
If integer $k\geq 2^m$, then we have $e^{im\theta_n}[0,k]\subseteq f_{k,m}(W^k)$ for all $n\in\mathbb{Z}$.
\end{lemma}

\begin{proof}
    For any $n\in\mathbb{N}$ and $x\in C$, we have $r^n x\in C$.
    So $(1+i r^n)x, (r^n+i)x\in W$.
    It follows from Theorem \ref{2^m} that $f_{k,m}(C^k)=[0,k]$ for $k\geq 2^m$.
    Then for $n\in\mathbb{N}$, taking the definition of $f_{k,m}$ into consideration, we have
  \begin{align*}
    (1+r^{2n})^{\frac{m}{2}}e^{im\theta_n}[0,k]=(1+i r^n)^m f_{k,m}(C^k)\subseteq f_{k,m}(W^k),\\  (1+r^{2n})^{\frac{m}{2}}e^{im\theta_{-n}}[0,k]=(r^n+i)^m f_{k,m}(C^k)\subseteq f_{k,m}(W^k).
  \end{align*}
   Hence, for each $n\in\mathbb{Z}$, we obtain
  \begin{equation*}
    e^{im\theta_n}[0,k]\subseteq \left(1+r^{2|n|}\right)^{\frac{m}{2}}e^{im\theta_n}[0,k]\subseteq f_{k,m}(W^k).
  \end{equation*}
\end{proof}

The following lemma shows that $f_{k,m}(W^k)$ has some symmetric property.

\begin{lemma}\label{lemfkmW}
Let $k$ and $m$ be positive integers.
Then we have $i^m \overline{f_{k,m}(W^k)}=f_{k,m}(W^k)$.
\end{lemma}
\begin{proof}
  If $z=x+iy\in W$ with $x,y\in C$, then $i\overline{z}=y+ix\in W$.
  Let $(z_1, z_2, \cdots, z_k)\in W^k$.
  On one hand, by the definition of $f_{k,m}$, we have $i^m\overline{f_{k,m}(z_1, z_2, \cdots, z_k)}=f_{k,m}(i\overline{z_1}, i\overline{z_2}, \cdots, i\overline{z_k})\in f_{k,m}(W^k)$.
  So $i^m \overline{f_{k,m}(W^k)}\subseteq f_{k,m}(W^k)$.
  On the other hand, we have
  \begin{align*}
   f_{k,m}(z_1, z_2, \cdots, z_k)&=i^mf_{k,m}(-i z_1, -i z_2, \cdots, -i z_k)\\
   &=i^m\overline{f_{k,m}(i\overline{z_1}, i\overline{z_2}, \cdots, i\overline{z_k})}\in i^m \overline{f_{k,m}(W^k)}.
  \end{align*}
  So $f_{k,m}(W^k)\subseteq i^m \overline{f_{k,m}(W^k)}$.
  The proof is completed.
\end{proof}

In order to prove Theorem 6.1, we discuss the range of $f_{k,m}(W^k)$ in four cases: $m\equiv 0\pmod 4$, $m\equiv 1\pmod 4$, $m\equiv 2\pmod 4$ and $m\equiv 3\pmod 4$.

\begin{lemma}
  If $m\equiv 0\pmod 4$, then we have $\frac{k}{24m^2}S\subseteq f_{3k,m}(W^{3k})$ for integer $k\geq 2^m$.
\end{lemma}

\begin{proof}
Note that $m\ge 4$, $\theta_0=\frac{\pi}{4}$, and $\theta_n\to 0$ as $n\to\infty$.
Then there is an integer $n_0\geq 0$ such that $m\theta_{n_0+1}\leq \frac{\pi}{2} < m \theta_{n_0}$.
According to Lemma \ref{lem3beijiao}, we have $\frac{\pi}{6}<m\theta_{n_0+1}\leq \frac{\pi}{2}$ and $\frac{\pi}{2} < m \theta_{n_0} \leq \frac{3\pi}{2}-mr^{3n_0+3}$.
It follows from Lemma \ref{inequality} that $\arctan x \leq x \leq\frac{4}{\pi}\arctan x$ for $x\in[0,1]$.
Then we obtain
\begin{equation*}
  \left(\frac{\pi}{6m}\right)^3<\theta_{n_0+1}^3 \leq r^{3n_0+3}\leq \left(\frac{4}{\pi}\theta_{n_0+1}\right)^3\leq\left(\frac{2}{m}\right)^3.
\end{equation*}
If $m\theta_{n_0} < \frac{\pi}{2}+mr^{3n_0+3}$, then using Lemma \ref{lem3beijiao} we obtain
\begin{equation*}
  m\theta_{n_0-1} \leq 3m\theta_{n_0}-mr^{3n_0}< \frac{3\pi}{2}-24mr^{3n_0+3} < \frac{3\pi}{2}-mr^{3n_0+3}.
\end{equation*}
By Lemma \ref{inequality}, we have $\arctan(3x) \geq \frac{3}{2}\arctan x$ for $x\in[0,1]$.
Then $m\theta_{n_0-1} \geq \frac{3}{2}m\theta_{n_0}>\frac{3\pi}{4}=\frac{\pi}{2}+\frac{\pi}{4}$ by letting $x=r^{n_0}$.
Note that $mr^{3n_0+3}\leq m(\frac{2}{m})^3=\frac{8}{m^2}\leq \frac{1}{2}< \frac{\pi}{4}$.
Hence, if $m\theta_{n_0} < \frac{\pi}{2}+mr^{3n_0+3}$, then $\frac{\pi}{2}+mr^{3n_0+3} \leq m\theta_{n_0-1} \leq \frac{3\pi}{2}-mr^{3n_0+3}$.
Let
\begin{equation*}
  n_1=\begin{cases}
        n_0-1, & \mbox{if } m\theta_{n_0} < \frac{\pi}{2}+mr^{3n_0+3},\\
        n_0, & \mbox{otherwise}.
      \end{cases}
\end{equation*}
It follows that
\begin{equation*}
  \frac{\pi}{2}+\frac{\pi^3}{216m^2}< \frac{\pi}{2}+mr^{3n_0+3} \leq  m\theta_{n_1} \leq  \frac{3\pi}{2}-mr^{3n_0+3}< \frac{3\pi}{2}-\frac{\pi^3}{216m^2}.
\end{equation*}
So $\frac{\pi^3}{216m^2}<\frac{3\pi}{2}-m\theta_{n_1}<\pi-\frac{\pi^3}{216m^2}$.
Again, by Lemma \ref{inequality}, we have $\sin x\geq \frac{2}{\pi}x$ for $x\in[0,1]$.
Note that $\frac{\pi^3}{216m^2}<1$.
We derive
\begin{equation*}
  -\cos(m\theta_{n_1})=\sin\left(\frac{3\pi}{2}-m\theta_{n_1}\right)>
  \sin\frac{\pi^3}{216m^2}\ge\frac{\pi^2}{108m^2}.
\end{equation*}
We have $m\theta_{-n_1}=m(\frac{\pi}{2}-\theta_{n_1})\equiv-m\theta_{n_1}\pmod {2\pi}$ since $m\equiv 0\pmod 4$.
Then
\begin{equation*}
  e^{im\theta_{n_1}}+e^{im\theta_{-n_1}}=2\cos(m\theta_{n_1})
      \leq -\frac{\pi^2}{54m^2}\leq -\frac{1}{6m^2}.
\end{equation*}
Let $k\geq 2^m$ and $t=\frac{k}{6m^2}<k$.
Using Lemma \ref{lemK}, we see that
\begin{align*}
  f_{2k,m}(W^{2k})&=f_{k,m}(W^{k})+ f_{k,m}(W^{k})\supseteq \left(e^{im\theta_{n_1}}[0,k]+e^{im\theta_{-n_1}}[0,k]\right)\cup [0,2k]\\
     &\supseteq(e^{im\theta_{n_1}}+e^{im\theta_{-n_1}})[0,k]\cup[0,2k] \supseteq[-t,0]\cup[0,2k]=[-t,2k].
\end{align*}
Moreover, we have $\frac{1}{2}<\sin(m\theta_{n_0+1}) \leq 1$ and $0\leq \cos(m\theta_{n_0+1}) < \frac{\sqrt{3}}{2}$.
Then
\begin{align*}
  f_{3k,m}(W^{3k})&=f_{2k,m}(W^{2k})+f_{k,m}(W^{k})\supseteq [-t,2k]+e^{im\theta_{n_0+1}}[0,k]\\
     &\supseteq\left\{a+b\cos(m\theta_{n_0+1})+ib\sin(m\theta_{n_0+1}): a\in[-t,t],b\in\left[0,\frac{\sqrt{3}}{3}t\right]\right\}\\
     &\supseteq \left\{a+ib\sin(m\theta_{n_0+1}):a\in\left[-\frac{t}{2},\frac{t}{2}\right],b\in\left[0,\frac{\sqrt{3}}{3}t\right]\right\}\\
     &\supseteq \left[-\frac{t}{2},\frac{t}{2}\right]+i\left[0,\frac{\sqrt{3}}{6}t\right]\supseteq \left[-\frac{1}{4}t,\frac{1}{4}t\right]+i\left[0,\frac{1}{4}t\right].
\end{align*}
Using Lemma \ref{lemfkmW}, we obtain $\overline{f_{3k,m}(W^{3k})}=f_{3k,m}(W^{3k})$.
It follows that $\frac{t}{4}S=\frac{k}{24m^2}S\subseteq f_{3k,m}(W^{3k})$.
\end{proof}

\begin{lemma}
  If $m\equiv 3\pmod 4$, then we have $\frac{\sqrt{2}k}{4}S\subseteq f_{4k,m}(W^{4k})$ for integer $k\geq 2^m$.
\end{lemma}

\begin{proof}
Since $m\equiv 3\pmod 4$, we have $-i\overline{f_{k,m}(W^k)}=f_{k,m}(W^k)$ from Lemma \ref{lemfkmW}.
It follows that $-i[0,k]=-if_{k,m}(C^k)\subseteq f_{k,m}(W^k)$.
Since $k\geq 2^m$, we have
\begin{equation*}
  f_{2k,m}(W^{2k})= f_{k,m}(W^{k})+f_{k,m}(W^{k}) \supseteq [0,k]-i[0,k].
\end{equation*}
Note that $m\ge 3$, $\theta_0=\frac{\pi}{4}$ and $\theta_n\to 0$ as $n\to\infty$.
We can see that $m\theta_{n_0+1}\leq \frac{\pi}{3} < m \theta_{n_0}$ for some integer $n_0\geq 0$.
By Lemma \ref{lem3beijiao}, we have $\frac{\pi}{3} < m\theta_{n_0} < 3m\theta_{n_0+1} \leq \pi$.
So $-\frac{5\pi}{12}<-\frac{\pi}{4} < \frac{3\pi}{4}-m\theta_{n_0}<\frac{5\pi}{12}$.
Since $m\equiv 3\pmod 4$, we obtain $m\theta_{-n}=m(\frac{\pi}{2}-\theta_{n})\equiv\frac{3\pi}{2}-m\theta_{n}\pmod {2\pi}$.
Then we derive that $e^{im\theta_{n_0}}+e^{im\theta_{-n_0}}=2\cos(\frac{3\pi}{4}-m\theta_{n_0})e^{i\frac{3\pi}{4}}$.
Notice that $\cos(\frac{3\pi}{4}-m\theta_{n_0})\geq \cos \frac{5\pi}{12} >\frac{1}{4}$.
From Lemma \ref{lemK}, we obtain
\begin{align*}
  f_{2k,m}(W^{2k})&= f_{k,m}(W^{k})+f_{k,m}(W^{k}) \supseteq e^{im\theta_{n_0}}[0,k]+e^{im\theta_{-n_0}}[0,k]\\
     &\supseteq \left(e^{im\theta_{n_0}}+e^{im\theta_{-n_0}}\right)[0,k]\supseteq e^{i\frac{3\pi}{4}}\left[0,\frac{k}{2}\right].
\end{align*}
Consequently, we have
\begin{align*}
  f_{4k,m}(W^{4k})&=f_{2k,m}(W^{2k})+f_{2k,m}(W^{2k})\\
     &\supseteq\left[0,\frac{\sqrt{2}}{2}k\right]-i\left[0,\frac{\sqrt{2}}{2}k\right]+\frac{k}{2}e^{i\frac{3\pi}{4}}\\
     &=\left[0,\frac{\sqrt{2}}{2}k\right]-\frac{\sqrt{2}}{4}k+i\left(\left[-\frac{\sqrt{2}}{2}k,0\right]+\frac{\sqrt{2}}{4}k\right)
     =\frac{\sqrt{2}k}{4}S.
\end{align*}
\end{proof}

\begin{lemma}
If $m\equiv 1\pmod 4$, then we have $\frac{2k}{25}(\frac{2}{3})^m S\subseteq f_{4k,m}(W^{4k})$ for integer $k\geq 2^m$.
\end{lemma}

\begin{proof}
Since $m\equiv 1\pmod 4$, it follows from Lemma \ref{lemfkmW} that $i\overline{f_{k,m}(W^k)}=f_{k,m}(W^k)$.
Then we obtain $if_{k,m}(C^k)\subseteq f_{k,m}(W^k)$.
If $k\geq 2^m$, then we have
\begin{equation*}
  f_{2k,m}(W^{2k})= f_{k,m}(W^{k})+f_{k,m}(W^{k}) \supseteq[0,k]+i[0,k].
\end{equation*}
Note that $m\ge 5$, $\theta_0=\frac{\pi}{4}$ and $\theta_n\to 0$ as $n\to\infty$.
We derive $m\theta_{n_0+1}\leq \pi < m \theta_{n_0}$ for some integer $n_0\geq 0$.
Using the fact that $m\equiv 1\pmod 4$, we obtain
$m\theta_{-n}=m(\frac{\pi}{2}-\theta_{n})\equiv\frac{5\pi}{2}-m\theta_{n}\pmod {2\pi}$.
Then we prove the lemma in three cases according to the range of $m\theta_{n_0+1}$.
  \begin{enumerate}
    \item $0 < m\theta_{n_0+1}\leq \frac{7\pi}{12}$.

    By Lemma \ref{lem3beijiao}, we have $m\theta_{n_0+1}>\frac{1}{3}m\theta_{n_0}>\frac{\pi}{3}$ and $\pi < m\theta_{n_0} \leq 3m\theta_{n_0+1}-mr^{3n_0+3}\leq \frac{7\pi}{4}-mr^{3n_0+3}$.
    Then $e^{im\theta_{n_0}}+e^{im\theta_{-n_0}}= 2\cos(m\theta_{n_0}-\frac{5\pi}{4})e^{i\frac{5\pi}{4}}$.
    From Lemma \ref{inequality}, we have $x\geq\arctan x$ for $x\in[0,1]$, then $mr^{3n_0+3}\geq m\theta_{n_0+1}^3>m\left(\frac{\pi}{3m}\right)^3
     =\frac{\pi^3}{27m^2}$.
    Note that $\frac{\pi^3}{27m^2}< \frac{\pi}{4}$.
    Then
    \begin{equation*}
      \frac{3\pi}{4}+\frac{\pi^3}{27m^2}<\pi< m\theta_{n_0}<\frac{7\pi}{4}-\frac{\pi^3}{27m^2}.
    \end{equation*}
    It follows that $\frac{\pi^3}{27m^2}< m\theta_{n_0}-\frac{3\pi}{4} < \pi-\frac{\pi^3}{27m^2}$.
    Again, by Lemma \ref{inequality}, we have $\sin x\geq \frac{2}{\pi}x$ for $x\in[0,1]$.
    Since $\frac{\pi^3}{27m^2}<1$, we obtain $\cos\left(m\theta_{n_0}-\frac{5\pi}{4}\right)=\sin\left(m\theta_{n_0}-\frac{3\pi}{4}\right)>
     \sin\frac{\pi^3}{27m^2} \geq \frac{2\pi^2}{27m^2} > \frac{2}{3m^2}$.
    It follows from Lemma \ref{lemK} that
    \begin{align*}
     f_{2k,m}(W^{2k})&= f_{k,m}(W^{k})+f_{k,m}(W^{k}) \supseteq e^{im\theta_{n_0}}[0,k]+e^{im\theta_{-n_0}}[0,k]\\
     &\supseteq \left(e^{im\theta_{n_0}}+e^{im\theta_{-n_0}}\right)[0,k]\supseteq e^{i\frac{5\pi}{4}}\left[0,\frac{4k}{3m^2}\right].
    \end{align*}
    Let $t_1=\frac{4k}{3m^2}$.
    Then $\sqrt{2}t_1<k$.
    We derive
    \begin{align*}
     f_{4k,m}(W^{4k})&= f_{2k,m}(W^{2k})+f_{2k,m}(W^{2k})\\
     &\supseteq [0,\sqrt{2}t_1]+i[0,\sqrt{2}t_1]+e^{i\frac{5\pi}{4}}t_1\\
     &=[0,\sqrt{2}t_1]-\frac{\sqrt{2}}{2}t_1+i\left([0,\sqrt{2}t_1]-\frac{\sqrt{2}}{2}t_1\right)
     =\frac{\sqrt{2}t_1}{2}S=\frac{2\sqrt{2}k}{3m^2}S.
    \end{align*}

    \item $\frac{5\pi}{6}\leq m\theta_{n_0+1}\leq \pi$.

    It is not difficult to see that $e^{im\theta_{n_0+1}}+e^{im\theta_{-(n_0+1)}}= 2\cos(m\theta_{n_0+1}-\frac{5\pi}{4})e^{i\frac{5\pi}{4}}$. Moreover, we have  $\cos(\frac{5\pi}{4}-m\theta_{n_0+1})\geq\cos\frac{5\pi}{12}>\frac{1}{4}$.
    From Lemma \ref{lemK}, we obtain
    \begin{align*}
     f_{2k,m}(W^{2k})&= f_{k,m}(W^{k})+f_{k,m}(W^{k})\supseteq e^{im\theta_{n_0+1}}[0,k]+e^{im\theta_{-(n_0+1)}}[0,k]\\
     &\supseteq \left(e^{im\theta_{n_0+1}}+e^{im\theta_{-(n_0+1)}}\right)[0,k]\supseteq e^{i\frac{5\pi}{4}}\left[0,\frac{k}{2}\right].
    \end{align*}
    Then
    \begin{align*}
     f_{4k,m}(W^{4k})&= f_{2k,m}(W^{2k})+f_{2k,m}(W^{2k})\\
     &\supseteq \left[0,\frac{\sqrt{2}}{2}k\right]+i\left[0,\frac{\sqrt{2}}{2}k\right]+\frac{k}{2}e^{i\frac{5\pi}{4}}\\
     &=\left[0,\frac{\sqrt{2}}{2}k\right]-\frac{\sqrt{2}}{4}k+i\left(\left[0,\frac{\sqrt{2}}{2}k\right]-\frac{\sqrt{2}}{4}k\right)
     =\frac{\sqrt{2}k}{4}S.
    \end{align*}

    \item $\frac{7\pi}{12} < m\theta_{n_0+1} < \frac{5\pi}{6}$.

    It follows from Lemma \ref{inequality} that $\frac{7}{5}\arctan x \leq \arctan(\frac{3}{2}x) \leq  \frac{3}{2}\arctan x$ for $x\in[0,\frac{1}{3}]$.
    Then
    \begin{equation*}
      \frac{49\pi}{60}<\frac{7}{5}m\theta_{n_0+1}\leq m\arctan\left(\frac{3}{2}r^{n_0+1}\right)\leq \frac{3}{2}m\theta_{n_0+1} < \frac{5\pi}{4}.
    \end{equation*}
    Let $\beta=\arctan(\frac{3}{2}r^{n_0+1})$. Then $\arctan\left(\frac{2}{3}r^{-(n_0+1)}\right)=\frac{\pi}{2}-\beta$ and $0 < \frac{5\pi}{4}-m\beta < \frac{13\pi}{30}$.
    So we have
    $\cos(\frac{5\pi}{4}-m\beta)\in(\frac{1}{5},1)$.

    Let $x=\frac{2}{3}$, $y= r^{n_0+1}$ and  $\rho=\sqrt{\frac{4}{9}+r^{2n_0+2}}\in(\frac{2}{3},1)$.
    We can see that $\rho e^{i\beta}=x+iy\in W$ and $\rho e^{i(\frac{\pi}{2}-\beta)}=y+ix\in W$ since $x,y\in C$.
    Using the fact that $m\equiv1\pmod 4$, we obtain
    $m(\frac{\pi}{2}-\beta)\equiv\frac{5\pi}{2}-m\beta\pmod {2\pi}$.

    Let $k_1=\lfloor\frac{k}{3}\rfloor$.
    Since $k\geq 2^m$ and $m\geq 5$, we derive $1<\frac{\sqrt{2}k}{5}<\frac{k}{3}-1<k_1\leq\frac{k}{3}$.
    Then
    \begin{align*}
    2k_1\rho^m\cos\left(\frac{5\pi}{4}-m\beta\right) e^{i\frac{5\pi}{4}}&=k_1\rho^m e^{im\beta}+k_1\rho^m e^{i(\frac{5\pi}{2}-m\beta)}\\
    &\in f_{k_1,m}(W^{k_1})+f_{k_1,m}(W^{k_1})\subseteq f_{2k,m}(W^{2k}).
    \end{align*}
    Let $t_2=2k_1\rho^m\cos(\frac{5\pi}{4}-m\beta)$.
    Then $\frac{2\sqrt{2}k}{25}(\frac{2}{3})^m<\frac{2}{5}k_1(\frac{2}{3})^m<t_2<2k_1<\frac{\sqrt{2}}{2}k$,
    and we have
    \begin{align*}
    f_{4k,m}(W^{4k})&= f_{2k,m}(W^{2k})+f_{2k,m}(W^{2k})\\
    &\supseteq[0,k]+i[0,k]+t_2 e^{i\frac{5\pi}{4}}\\
    &=[0,k]-\frac{\sqrt{2}}{2}t_2+i\left([0,k]-\frac{\sqrt{2}}{2}t_2\right)\supseteq \frac{\sqrt{2}t_2}{2}S\supseteq \frac{2k}{25}\left(\frac{2}{3}\right)^mS.
    \end{align*}
  \end{enumerate}
Through the discussion of the above three cases, since $\frac{\sqrt{2}}{4}>\frac{2\sqrt{2}}{3m^2}>\frac{2}{25}(\frac{2}{3})^m$ for $m\ge 5$,
we obtain that $\frac{2k}{25}(\frac{2}{3})^mS\subseteq f_{4k,m}(W^{4k})$.
\end{proof}

\begin{lemma}
    If $m\equiv 2\pmod 4$, then we have $\frac{12k}{125}(\frac{2}{3})^mS\subseteq f_{4k,m}(W^{4k})$ for integer  $k\geq 2^m$.
\end{lemma}

\begin{proof}
Since $m\equiv 2\pmod 4$, we obtain that $-\overline{f_{k,m}(W^k)}=f_{k,m}(W^k)$ from Lemma \ref{lemfkmW}.
Then we derive $-f_{k,m}(C^k)\subseteq f_{k,m}(W^k)$.
If $k\geq 2^m$, then we have
\begin{equation*}
  f_{k,m}(W^{k}) \supseteq [-k,0]\cup[0,k]= [-k,k].
\end{equation*}
Note that $m\ge 6$, $\theta_0=\frac{\pi}{4}$, and $\theta_n\to 0$ as $n\to\infty$.
Then we can see that  $m\theta_{n_0+1}\leq \frac{3\pi}{4} < m \theta_{n_0}$ and $m\theta_{n_1+1}\leq \pi < m \theta_{n_1}$ for some  $n_0,n_1\geq 0$.
Taking $m\equiv 2\pmod 4$ into consideration, we obtain that
$m\theta_{-n}=m(\frac{\pi}{2}-\theta_{n})\equiv 3\pi-m\theta_{n}\pmod {2\pi}$.
From Lemma \ref{lem3beijiao}, we have $\frac{\pi}{4}<\frac{1}{3}m\theta_{n_0}\le m\theta_{n_0+1}\le\frac{3\pi}{4}$.
So $\frac{\sqrt{2}}{2}\le\sin(m\theta_{n_0+1}) \leq 1$ and $-\frac{\sqrt{2}}{2}\leq \cos(m\theta_{n_0+1})\le\frac{\sqrt{2}}{2}$.
It follows from Lemma \ref{lemK} that
\begin{align*}
  f_{2k,m}(W^{2k})&=f_{k,m}(W^{k})+f_{k,m}(W^{k})\supseteq [-k,k]+e^{im\theta_{n_0+1}}[0,k]\\
     &\supseteq\left\{a+b\cos(m\theta_{n_0+1})+ib\sin(m\theta_{n_0+1}): a\in[-k,k],b\in\left[0,\frac{\sqrt{2}}{2}k\right]\right\}\\
     &\supseteq \left\{a+ib\sin(m\theta_{n_0+1}):a\in\left[-\frac{k}{2},\frac{k}{2}\right],b\in\left[0,\frac{\sqrt{2}}{2}k\right]\right\}\\
     &\supseteq \left[-\frac{k}{2},\frac{k}{2}\right]+i\left[0,\frac{k}{2}\right].
\end{align*}
Again, by Lemma \ref{lem3beijiao}, we have  $\frac{\pi}{3}<m\theta_{n_1+1}\le\pi$.
We next prove the lemma in two cases according to the range of $m\theta_{n_1+1}$.
   \begin{enumerate}
    \item $\frac{\pi}{3}< m\theta_{n_1+1}\leq \frac{2\pi}{3}$.

    From Lemma \ref{lem3beijiao}, we have  $\pi<m\theta_{n_1}\le 2\pi-m r^{3n_1+3}$.
    We next consider the range of $m\theta_{n_1}$ in two subcases.
     \begin{enumerate}
     \item $\pi < m\theta_{n_1} \leq \frac{7\pi}{6}$.

     Since $m\theta_0\geq\frac{3\pi}{2}>m\theta_{n_1}$, we have $n_1\geq 1$.
     We obtain $\frac{5}{4}\arctan x \leq \arctan(\frac{3}{2}x) \leq  \frac{3}{2}\arctan x$ for $x\in[0,1]$ from Lemma \ref{inequality}.
     So we have $\frac{5\pi}{4}<
     \frac{5}{4}m\theta_{n_1}\leq m\arctan\left(\frac{3}{2}r^{n_1}\right)\leq \frac{3}{2}m\theta_{n_1} \leq \frac{7\pi}{4}$.
     Let $\beta_1=\arctan(\frac{3}{2}r^{n_1})$.
     Then we obtain that $\arctan(\frac{2}{3}r^{-n_1})=\frac{\pi}{2}-\beta_1$ and $-\frac{\pi}{4}\leq \frac{3\pi}{2}-m\beta_1 \leq\frac{\pi}{4}$.
     So $\cos(\frac{3\pi}{2}-m\beta_1)\in(\frac{\sqrt{2}}{2},1)$.

     Let $x_1=\frac{2}{3}$, $y_1= r^{n_1}$ and  $\rho_1=\sqrt{\frac{4}{9}+r^{2n_1}}\in(\frac{2}{3},1)$.
     Then $\rho_1 e^{i\beta_1}=x_1+iy_1\in W$ and $\rho_1 e^{i(\frac{\pi}{2}-\beta_1)}=y_1+ix_1\in W$ since $x_1,y_1\in C$.
     Using the fact that $m\equiv 2\pmod 4$, we have
     $m(\frac{\pi}{2}-\beta_1)\equiv 3\pi-m\beta_1\pmod {2\pi}$.

     Let $k_1=\big\lfloor\frac{k}{4}\big\rfloor$.
     Since $k\geq 2^m$ and $m\geq 6$, we derive $\frac{k}{5}<k_1\le\frac{k}{4}$ and
     \begin{align*}
     -2k_1\rho_1^m\cos\left(\frac{3\pi}{2}-m\beta_1\right)i &=k_1\rho_1^m e^{im\beta_1}+k_1\rho_1^m e^{i(3\pi-m\beta_1)}\\
     &\in f_{k_1,m}(W^{k_1})+f_{k_1,m}(W^{k_1})\subseteq f_{2k,m}(W^{2k}).
     \end{align*}

     Let $t_1=2k_1\rho_1^m\cos(\frac{3\pi}{2}-m\beta_1)$.
     Then $\frac{\sqrt{2}k}{5}(\frac{2}{3})^m \leq t_1\leq\frac{k}{2}$.
     It follows that
      \begin{align*}
      f_{4k,m}(W^{4k})&\supseteq \left(f_{2k,m}(W^{2k})+f_{2k,m}(W^{2k})\right)\cup f_{2k,m}(W^{2k})\\
      &\supseteq \left(\left[-\frac{k}{2},\frac{k}{2}\right]+i\left[0,\frac{k}{2}\right]-t_1 i\right)\cup \left(\left[-\frac{k}{2},\frac{k}{2}\right]+i\left[0,\frac{k}{2}\right]\right)\\
      &\supseteq t_1 S\supseteq \frac{\sqrt{2}k}{5}\left(\frac{2}{3}\right)^mS.
      \end{align*}

      \item $\frac{7}{6}\pi< m\theta_{n_1} \leq 2\pi-mr^{3n_1+3}$.

     It follows from Lemma \ref{inequality} that $\arctan x \leq x \leq\frac{4}{\pi}\arctan x$ for $x\in[0,1]$. Then
     \begin{equation*}
     \frac{\pi^3}{27m^2}\leq m\theta_{n_1+1}^3\leq mr^{3n_1+3}\leq m\left(\frac{4}{\pi}\theta_{n_1+1}\right)^3
     \leq\frac{2^9}{27m^2}.
     \end{equation*}
     So we have $\pi+\frac{\pi^3}{27m^2}<\pi+\frac{\pi}{6}=\frac{7\pi}{6} < m\theta_{n_1} < 2\pi-\frac{\pi^3}{27m^2}$.
     Moreover, we can see that $e^{im\theta_{n_1}}+e^{im\theta_{-n_1}}=-2i\cos(\frac{3\pi}{2}-m\theta_{n_1})$.
     Again, by Lemma \ref{inequality}, we have $\sin x\geq \frac{2}{\pi}x$ for $x\in[0,1]$.
     Note that $\frac{\pi^3}{27m^2}<1$.
     We derive $\cos\left(\frac{3\pi}{2}-m\theta_{n_1}\right)=-\sin(m\theta_{n_1})>\sin\frac{\pi^3}{27m^2}\geq \frac{2\pi^2}{27m^2}>\frac{2}{3m^2}$.
     From Lemma \ref{lemK}, we have
     \begin{align*}
     f_{2k,m}(W^{2k})&= f_{k,m}(W^{k})+f_{k,m}(W^{k})\supseteq e^{im\theta_{n_1}}[0,k]+e^{im\theta_{-n_1}}[0,k]\\
     &\supseteq \left(e^{im\theta_{n_1}}+e^{im\theta_{-n_1}}\right)[0,k]\supseteq -i\left[0,\frac{4k}{3m^2}\right].
     \end{align*}
     Since $\frac{4k}{3m^2}<\frac{k}{4}$, we obtain
     \begin{align*}
     f_{4k,m}(W^{4k})&=f_{2k,m}(W^{2k})+f_{2k,m}(W^{2k})\\
     &\supseteq \left[-\frac{k}{2},\frac{k}{2}\right]+i\left[0,\frac{k}{2}\right]+i\left[-\frac{4k}{3m^2},0\right]\supseteq\frac{4k}{3m^2}S.
     \end{align*}
     \end{enumerate}
     We have $\frac{\sqrt{2}}{5}(\frac{2}{3})^m<\frac{4}{3m^2}$ since $m\geq 6$.
     Through the discussion of the above two subcases, we derive $\frac{\sqrt{2}k}{5}(\frac{2}{3})^mS\subseteq f_{4k,m}(W^{4k})$.

    \item $\frac{2\pi}{3} < m\theta_{n_1+1}\leq \pi$:

     Using Lemma \ref{inequality}, we have $\frac{9}{5}\arctan x \leq \arctan (2x)\leq 2\arctan x-x^3$ for $x\in[0,\frac{1}{3}]$ and $\arctan x \leq x$ for $x\in[0,1]$.
     Then we can see that
     \begin{equation*}
      \frac{6\pi}{5}<\frac{9}{5}m\theta_{n_1+1}\leq m\arctan(2r^{n_1+1}) \leq 2m\theta_{n_1+1}-m r^{3n_1+3}\leq 2\pi-mr^{3n_1+3}
     \end{equation*}
     and $mr^{3n_1+3}\geq m\theta_{n_1+1}^3>m\left(\frac{2\pi}{3m}\right)^3=\frac{8\pi^3}{27m^2}$.
     Let $\beta_2=\arctan(2r^{n_1+1})$.
     Since $\frac{8\pi^3}{27m^2}<\frac{\pi}{5}$, we have
     \begin{equation*}
     \pi+\frac{8\pi^3}{27m^2}<\pi+\frac{\pi}{5}=\frac{6\pi}{5}
     <m\beta_2<2\pi-\frac{8\pi^3}{27m^2}.
     \end{equation*}
     It follows from Lemma \ref{inequality} that $\sin x\geq \frac{2}{\pi}x$ for $x\in[0,1]$.
     Note that $\frac{8\pi^3}{27m^2}<1$.
     We derive
     \begin{equation*}
       1\geq \cos\left(\frac{3\pi}{2}-m\beta_2\right)=-\sin(m\beta_2)
       >\sin\frac{8\pi^3}{27m^2}\geq \frac{16\pi^2}{27m^2}>\frac{5}{m^2}.
     \end{equation*}
     Let $x_2=1$, $y_2=2r^{n_1+1}$ and $\rho_2=\sqrt{1+4r^{2n_1+2}}\in(1,\frac{5}{4})$.
     Then $\rho_2 e^{i\beta_2}=x_2+iy_2\in W$ and $\rho_2 e^{i(\frac{\pi}{2}-\beta_2)}=y_2+ix_2\in W$ since $x_2,y_2\in C$.

     Let $k_2=\big\lfloor\frac{1}{4}(\frac{4}{5})^mk\big\rfloor\leq\frac{1}{4}(\frac{4}{5})^mk$. Then $k_2\geq \frac{1}{4}(\frac{4}{5})^mk-1\geq\frac{4}{25}(\frac{4}{5})^mk\geq 1$ since $m\geq 6$.
     Using the fact that $m\equiv 2\pmod 4$, we have
     $m(\frac{\pi}{2}-\beta_2)\equiv 3\pi-m\beta_2\pmod {2\pi}$.
     Then
     \begin{align*}
     -2k_2\rho_2^m\cos\left(\frac{3\pi}{2}-m\beta_2\right)i &=k_2\rho_2^m e^{im\beta_2}+k_2\rho_2^m e^{i(3\pi-m\beta_2)}\\
     &\in f_{k_2,m}(W^{k_2})+f_{k_2,m}(W^{k_2})\subseteq f_{2k,m}(W^{2k}).
     \end{align*}
     Let $t_2=2k_2\rho_2^m\cos(\frac{3\pi}{2}-m\beta_2)$.
     Then $\frac{8k}{5m^2}(\frac{4}{5})^{m}=\frac{10}{m^2}\cdot \frac{4}{25}(\frac{4}{5})^{m}k\leq t_2\leq\frac{k}{2}$.
     Consequently,
      \begin{align*}
      f_{4k,m}(W^{4k})&= \left(f_{2k,m}(W^{2k})+f_{2k,m}(W^{2k})\right)\cup f_{2k,m}(W^{2k})\\
      &\supseteq \left(\left[-\frac{k}{2},\frac{k}{2}\right]+i\left[0,\frac{k}{2}\right]-t_2 i\right)\cup \left(\left[-\frac{k}{2},\frac{k}{2}\right]+i\left[0,\frac{k}{2}\right]\right)\\
      &\supseteq t_2 S\supseteq \frac{8k}{5m^2}\left(\frac{4}{5}\right)^{m}S.
      \end{align*}
  \end{enumerate}
  We have  $\frac{12}{125}(\frac{2}{3})^m <\min\{\frac{\sqrt{2}}{5}(\frac{2}{3})^m,\frac{8}{5m^2}(\frac{4}{5})^{m}\}$ since $m\geq 6$.
  Through the discussion of the above two cases, we derive $\frac{12k}{125}(\frac{2}{3})^mS\subseteq f_{4k,m}(W^{4k})$.
\end{proof}

We are now ready to prove Theorem 6.1.

\begin{proof}
  We have $\frac{1}{100}(\frac{2}{3})^{m}< \frac{1}{24m^2}$, $\frac{1}{100}(\frac{2}{3})^{m}<\frac{\sqrt{2}}{4}$, $\frac{1}{100}(\frac{2}{3})^{m}<\frac{2}{25}(\frac{2}{3})^{m}$ and $\frac{1}{100}(\frac{2}{3})^{m}<\frac{12}{125}(\frac{2}{3})^{m}$ since $m\geq 3$.
  Combining Lemmas 6.6-6.9, we can see that $\frac{k}{100}(\frac{2}{3})^{m}S\subseteq f_{4k,m}(W^{4k})$ for $m\geq 3$ and $k\geq 2^m$.
\end{proof}

\section{$p$-adic Cantor set}

Let $m\geq 1$ be an integer and $p$ be a prime throughout this section.
By the Fundamental Theorem of Arithmetic, any $x\in \mathbb{Q}\backslash\{0\}$ can be written as the form $x=p^{v_p(x)}\frac{m}{n}$, where $m$ and $n$ are integers, $p\nmid mn$ and $v_p(x)\in\mathbb{Z}$.
We adopt the convention that $v_p(0)=\infty$.
Define
\begin{equation*}|x|_p=\left\{
\begin{aligned}
&p^{-v_p(x)}, & x\ne 0,\\
&0 , & x=0.
\end{aligned}\right.
\end{equation*}
It is well known that $|\cdot|_p$ is a norm on $\mathbb{Q}$.
The completion of $\mathbb{Q}$ under the norm $|\cdot|_p$ is denoted by $\mathbb{Q}_p$.
Let $N=\max\{-v_p(x),0\}$.
We can write $x\in\mathbb{Q}_p$ as
\begin{equation*}
  x=\sum_{n=-N}^{\infty} a_n p^n ,\ a_n\in\{0,1,\cdots,p-1\}.
\end{equation*}
The set of $p$-adic integers $\mathbb{Z}_p=\{x\in \mathbb{Q}_p: |x|_p\leq 1\}$ is the closure of $\mathbb{Z}$ in $\mathbb{Q}_p$ and it is compact.
We can write $x\in\mathbb{Z}_p$ as
\begin{equation*}
  x=\sum_{n=0}^{\infty} a_n p^n ,\ a_n\in\{0,1,\cdots,p-1\}.
\end{equation*}
Let $\mu_H$ denote the Haar measure on $\mathbb Z_p$.
Then $\mu_H(a+p^N \mathbb{Z}_p)=\frac{1}{p^N}$.
We refer the reader to \cite{Koblitz} for the general theory of $p$-adic numbers.

Let $\gamma\in p\mathbb{Z}_p$ and $2|\gamma|_p<1$.
Define
\begin{equation*}
  \mathcal{C}_\gamma=\mathcal{C}_{p,\gamma}=\left\{\sum_{n=0}^{\infty}a_n \gamma^n, a_n\in\{0,\gamma-1\}\right\}\subseteq\mathbb{Z}_p.
\end{equation*}
We call $\mathcal{C}_\gamma$ the $p$-adic Cantor set and we have $\mu_H(\mathcal{C}_\gamma)=0$.

In this section, we consider a similar Waring's problem on $p$-adic Cantor set $\mathcal{C}_\gamma=\mathcal{C}_{p,\gamma}$ with $\gamma\in p\mathbb{Z}_p$.
More precisely, for integer $m\ge 1$, is there a positive integer $k$ such that every element in $\mathbb{Z}_p$ can be written as $x_1^m+x_2^m+\cdots+x_k^m$ with $x_j\in C_\gamma$ and $1\leq j\leq k$?

We first consider the sum of elements in $\mathcal{C}_\gamma$ and obtain the following result.
\begin{theorem}
  Let $p$ be a prime, $\gamma\in p\mathbb{Z}_p$ and $2|\gamma|_p<1$.
  Then
  $\mathcal{G}_\gamma=\frac{1}{|\gamma|_p}-1$ is the smallest positive integer $k$ that satisfies
  \begin{equation*}
    \mathbb{Z}_p=\{x_1+x_2+\cdots+x_k:x_i\in \mathcal{C}_\gamma, 1\leq i\leq k\}.
  \end{equation*}
\end{theorem}
\begin{proof}
  Assume $\gamma=p^u \gamma_1$ and $p\nmid \gamma_1$, then $\mathcal{G}_\gamma=p^u-1$.
  Note that $\mathbb{Z}_p=\left\{\sum\limits_{n=0}^{\infty}b_n (p^u)^n:b_n\in\{0,1,\cdots,\mathcal{G}_r\}\right\}$.
  Since $p\nmid \gamma_1$, the residue class of module $p^u$ of the set $\left\{b_n \gamma_1^n: b_n\in \{0,1,\cdots,\mathcal{G}_\gamma\}\right\}$ is $\{0,1,\cdots,\mathcal{G}_\gamma\}$ for $n\geq 0$.
  Then
  \begin{equation*}
    \mathbb{Z}_p=\left\{\sum\limits_{n=0}^{\infty}b_n\gamma_1^n(p^u)^n:b_n\in\{0,1,\cdots,\mathcal{G}_r\}\right\}
    =\left\{\sum\limits_{n=0}^{\infty}b_n\gamma^n:b_n\in\{0,1,\cdots,\mathcal{G}_r\}\right\}.
  \end{equation*}
  Hence, any $x\in\mathbb{Z}_p$ can be written as $x=\sum\limits_{n=0}^{\infty}b_n\gamma^n$ with $b_n\in \{0,1,\cdots,\mathcal{G}_\gamma\}$
  and $(\gamma-1)x=\sum\limits_{n=0}^{\infty}b_n(\gamma-1)\gamma^n$.
  Let
  \begin{equation*}
   b_{i,n}=\begin{cases}
   \gamma-1,& b_n\geq i, \\
   0,& b_n<i,
   \end{cases}\
   x_i=\sum\limits_{n=0}^{\infty}b_{i,n} \gamma^n,\ 1\leq i\leq \mathcal{G}_\gamma.
   \end{equation*}
   It follows that $x_i\in \mathcal{C}_\gamma$ for $1\leq i \leq \mathcal{G}_\gamma$ and $(\gamma-1)x=x_1+x_2+\cdots+x_{\mathcal{G}_\gamma}$.
   Since $p|\gamma$, we have $p\nmid\gamma-1$.
   So $\mathbb{Z}_p=(\gamma-1)\mathbb{Z}_p=\{x_1+x_2+\cdots+x_{\mathcal{G}_\gamma}:x_i\in \mathcal{C}_\gamma, 1\leq i\leq \mathcal{G}_\gamma\}$.

   We next prove that $\mathbb{Z}_p\neq \{x_1+x_2+\cdots+x_{\mathcal{G}_\gamma-1}:x_i\in \mathcal{C}_\gamma, 1\leq i\leq \mathcal{G}_\gamma-1\}$.
   We have
   \begin{equation*}
     \{x_1+x_2:x_1,x_2\in \mathcal{C}_\gamma\}\subseteq\{0,\gamma-1,2(\gamma-1)\}+p^u\mathbb{Z}_p
   \end{equation*}
   since $\mathcal{C}_\gamma\subseteq\{0,\gamma-1\}+p^u\mathbb{Z}_p$.
   Further, we obtain
   \begin{equation*}
     \{x_1+x_2+\cdots+x_{\mathcal{G}_\gamma-1}:x_i\in \mathcal{C}_\gamma, 1\leq i\leq \mathcal{G}_\gamma-1\}\subseteq\{0,\gamma-1,2(\gamma-1),\cdots,(\mathcal{G}_\gamma-1)(\gamma-1)\}+p^u\mathbb{Z}_p.
   \end{equation*}
   So $\mathcal{G}_\gamma\cdot  (\gamma-1)=(p^u-1)(\gamma-1)\notin\{x_1+x_2+\cdots+x_{\mathcal{G}_\gamma-1}:x_i\in \mathcal{C}_\gamma, 1\leq i\leq \mathcal{G}_\gamma-1\}$.
\end{proof}

Let $\mathcal{G}_\gamma(m)$ denote the smallest positive integer $k$ that satisfies
  \begin{equation*}
    \mathbb{Z}_p=\{x_1^m+x_2^m+\cdots+x_k^m:x_i\in \mathcal{C}_\gamma, 1\leq i\leq k\}.
  \end{equation*}
Then $\mathcal{G}_\gamma(1)=\mathcal{G}_\gamma=\frac{1}{|\gamma|_p}-1$ by the above theorem.
For the case $m\ge 2$, we need to consider $p=2$ and $p\ge 3$ separately.

  For $p=2$, we have the following result.

\begin{theorem}
  Suppose that $\gamma\in 2\mathbb{Z}_2$ and $2|\gamma|_2<1$. For any integer $m\geq 2$, there is a positive integer  $k\leq\frac{1}{|\gamma|_2}\left(\frac{1}{|\gamma m|_2}+1\right)-2$, such that
  \begin{equation*}
    \mathbb{Z}_2=\{x_1^m+x_2^m+\cdots+x_k^m:x_i\in \mathcal{C}_\gamma, 1\leq i\leq k\}.
  \end{equation*}
  In other words, $\mathcal{G}_\gamma(m)\leq\frac{1}{|\gamma|_2}\left(\frac{1}{|\gamma m|_2}+1\right)-2$.
\end{theorem}

\begin{proof}
  Assume $\gamma=2^u \gamma_1\in 2\mathbb{Z}_2$ and $2\nmid \gamma_1$.
  We should have $u\geq 1$.
  Assume $m=2^v m_1$ and $2\nmid m_1$.
  We claim that for any $x\in\mathbb{Z}_2$ and $N\geq 2$, there are $y_i\in\{0,\gamma-1\}$ for $1\leq i\leq 2^{2u+v}-1$ and $x_{i,N}\in \bigg\{\sum\limits_{n=0}^{N-1}a_n \gamma^n:a_n\in\{0,\gamma-1\}\bigg\}$ for $1\leq i\leq 2^{u}-1$ with $x_{i,N}\equiv \gamma-1\pmod {2^u}$, such that
  \begin{equation}\label{thm7.2}
    x-\sum_{i=1}^{2^{2u+v}-1}y_i^m\equiv \sum_{i=1}^{2^{u}-1}x_{i,N}^m\pmod {2^{uN+v}}.
  \end{equation}
  We now prove it by induction on $N$.

  Suppose $N=2$.
  Let $x_{i,2}=\gamma-1$ for $1\leq i\leq 2^{u}-1$.
  We have $2\nmid(\gamma-1)$ since $2|\gamma$.
  Then the residue class of module $2^{2u+v}$ of the set $\left\{k (\gamma-1)^m: k\in \{0,1,\cdots,2^{2u+v}-1\}\right\}$ is $\{0,1,\cdots,2^{2u+v}-1\}$.
  So there is a $k_0\in\{0,1,\cdots,2^{2u+v}-1\}$ such that $x-\sum\limits_{i=1}^{2^{u}-1}x_{i,2}^m=x-(2^{u}-1)(\gamma-1)^m\equiv k_0(\gamma-1)^m \pmod {2^{2u+v}}$.
  Let
  \begin{equation*}
   y_i=\begin{cases}
   \gamma-1,& i\leq k_0, \\
   0,& i>k_0,
   \end{cases}\ \  1\leq i\leq 2^{2u+v}-1.
   \end{equation*}
  It follows that $y_i\in\{0,\gamma-1\}$ and $x-\sum\limits_{i=1}^{2^{2u+v}-1}y_i^m=x-k_0(\gamma-1)^m\equiv \sum\limits_{i=1}^{2^{u}-1}x_{i,2}^m\pmod {2^{2u+v}}$.
  Hence, (\ref{thm7.2}) is valid for $N=2$.

  Assume (\ref{thm7.2}) is valid for some $N\geq 2$.
  We next prove (\ref{thm7.2}) for $N+1$.
  If $z\in \{x_{i,N}:1\leq i\leq 2^u-1\}$, then we have $z\equiv \gamma-1\pmod {2^u}$.
  So $2\nmid z$.
  Note that
  \begin{equation*}
    (z+(\gamma-1)\gamma^N)^m=z^m+mz^{m-1}(\gamma-1)\gamma^N+\sum_{j=2}^{m}\binom{m}{j}z^{m-j}(\gamma-1)^j \gamma^{jN}.
  \end{equation*}
  For each $j\in\{2,3,\cdots,m\}$, we have $2j-3\geq v_2(j)$.
  It follows that
  \begin{equation*}
    Nuj-v_2(j)-u(N+1)= u\left((j-1)N-1\right)-v_2(j)\geq 2j-3-v_2(j)\geq 0.
  \end{equation*}
  Then for each $j\in\{2,3,\cdots,m\}$, we obtain
  \begin{align*}
     v_2\left(\binom{m}{j}z^{m-j}(\gamma-1)^j \gamma^{jN}\right)&= ujN+v_2(m)-v_2(j)+v_2\left(\binom{m-1}{j-1}\right)\\
     &\geq ujN+v-v_2(j)\geq u(N+1)+v.
  \end{align*}
  Note that $z^{m-1}\equiv (\gamma-1)^{m-1}\pmod {2^u}$.
  Then
   \begin{equation*}
     mz^{m-1}(\gamma-1)\gamma^N=m_1 z^{m-1}(\gamma-1)\gamma_1^N 2^{v+uN}\equiv m_1 (\gamma-1)^{m}\gamma_1^N 2^{v+uN}\pmod {2^{v+u(N+1)}}.
   \end{equation*}
   So we have $\left(z+(\gamma-1)\gamma^N\right)^m\equiv z^m+m_1(\gamma-1)^{m}\gamma_1^N 2^{v+uN}\pmod {2^{v+u(N+1)}}$.

   Using the fact that $2\nmid m_1(\gamma-1)^{m}\gamma_1^N$, we derive that the residue class of module $2^{u(N+1)+v}$ of the set $\left\{k m_1(\gamma-1)^{m}\gamma_1^N 2^{uN+v}: k\in\{0,1,\cdots,2^u-1\}\right\}$
   is the same as $\left\{k\cdot 2^{uN+v}:k\in\{0,1,\cdots,2^u-1\}\right\}$.
   Then there is a $k_1\in\{0,1,\cdots,2^u-1\}$ such that
   \begin{equation*}
     x-\sum\limits_{i=1}^{2^{2u+v}-1}y_i^m\equiv k_1 m_1(\gamma-1)^{m}\gamma_1^N 2^{uN+v}+\sum\limits_{i=1}^{2^{u}-1}x_{i,N}^m\pmod {2^{u(N+1)+v}}.
   \end{equation*}
   Let
   \begin{equation*}
   x_{i,N+1}=\begin{cases}
   x_{i,N},& i>k_1, \\
   x_{i,N}+(\gamma-1)\gamma^N,& i\leq k_1,
   \end{cases}\ \  1\leq i\leq 2^u-1.
   \end{equation*}
   Then we have $x_{i,N+1}\equiv x_{i,N} \equiv \gamma-1 \pmod {2^u}$ and
   \begin{equation*}
    \sum_{i=1}^{2^{u}-1}x_{i,N+1}^m\equiv k_1 m_1(\gamma-1)^{m}\gamma_1^N 2^{uN+v}+\sum_{i=1}^{2^{u}-1}x_{i,N}^m\equiv x-\sum_{i=1}^{2^{2u+v}-1}y_i^m\pmod {2^{u(N+1)+v}}.
  \end{equation*}
  This completes the proof of the claim.

  From the construction of $x_{i,N+1}$, we have $x_{i,M}\equiv x_{i,N}\pmod{2^{uN}}$ for each integer $M\geq N\geq 2$ and each $i\in\{0,1,\cdots 2^u-1\}$.
  Then we obtain $|x_{i,M}- x_{i,N}|_2\leq 2^{-uN}$.
  So $\{x_{i,N}\}_{N=2}^\infty$ is a Cauchy sequence in $\mathcal{C}_\gamma$ for each $i\in\{0,1,\cdots,2^u-1\}$.
  We can see that $\lim\limits_{N\rightarrow\infty}x_{i,N}=x_i\in\mathcal{C}_\gamma$ since $\mathcal{C}_\gamma$ is compact.
  Consequently, $\lim\limits_{N\rightarrow\infty} \sum\limits_{i=1}^{2^u-1}x_{i,N}^m=\sum\limits_{i=1}^{2^u-1}x_i^m$.
  Note that $\left|x-\sum\limits_{i=1}^{2^{2u+v}-1}y_i^m-\sum\limits_{i=1}^{2^{u}-1}x_{i,N}^m\right|_2\leq 2^{-(uN+v)}$ for any $N\geq 2$.
  We derive $\lim\limits_{N\rightarrow\infty}\left(x-\sum\limits_{i=1}^{2^{2u+v}-1}y_i^m- \sum\limits_{i=1}^{2^{u}-1}x_{i,N}^m\right)=0$.
  Hence, $x=\sum\limits_{i=1}^{2^{2u+v}-1}y_i^m+\sum\limits_{i=1}^{2^u-1}x_i^m$.
  It is not difficult to see that $2^{u}-1+2^{2u+v}-1=\frac{1}{|\gamma|_2}\left(\frac{1}{|\gamma|_2 |m|_2}+1\right)-2$.
  The proof is complete.
\end{proof}

We now consider the case $p\geq 3$.

\begin{theorem}
  Suppose that $\gamma\in p\mathbb{Z}_p$ and $2|\gamma|_p<1$, where $p\ge 3$ is a prime.
  For any integer $m\geq 2$, there is a positive integer $k\leq \frac{1}{|\gamma|_p}\left(\frac{1}{|m|_p}+1\right)-2$, such that
  \begin{equation*}
    \mathbb{Z}_p=\{x_1^m+x_2^m+\cdots+x_k^m:x_i\in \mathcal{C}_\gamma, 1\leq i\leq k\}.
  \end{equation*}
  In other words,  $\mathcal{G}_\gamma(m)\leq\frac{1}{|\gamma|_p}\left(\frac{1}{|m|_p}+1\right)-2$.
\end{theorem}

\begin{proof}
  Assume $\gamma=p^u \gamma_1\in p\mathbb{Z}_p$ and $p\nmid \gamma_1$.
  We should have $u\geq 1$.
  Assume $m=p^v m_1$ and $p\nmid m_1$.
  We claim that for any $x\in\mathbb{Z}_p$ and any $N\geq 1$, there are $y_i\in \{0,\gamma-1\}$ for $1\leq i\leq p^{u+v}-1$ and $x_{i,N}\in \left\{\sum\limits_{n=0}^{N-1}a_n \gamma^n:a_n\in\{0,\gamma-1\}\right\}$ for $1\leq i\leq p^{u}-1$ with $x_{i,N}\equiv \gamma-1\pmod {p^u}$, such that
  \begin{equation}\label{thm7.3}
    x-\sum_{i=1}^{p^{u+v}-1}y_i^m\equiv \sum_{i=1}^{p^{u}-1}x_{i,N}^m\pmod {p^{uN+v}}.
  \end{equation}
  We now prove it by induction on $N$.

  Suppose $N=1$.
  Let $x_{i,1}=\gamma-1$ for $1\leq i\leq p^{u}-1$.
  Since $p|\gamma$, we have $p\nmid(\gamma-1)$.
  Then the residue class of module $p^{u+v}$ of the set $\left\{k (\gamma-1)^m: k\in \{0,1,\cdots,p^{u+v}-1\}\right\}$ is $\{0,1,\cdots,p^{u+v}-1\}$.
  So there is a $k_0\in\{0,1,\cdots,p^{u+v}-1\}$ such that $x-\sum\limits_{i=1}^{p^{u}-1}x_{i,1}^m=x-(p^{u}-1)(\gamma-1)^m\equiv k_0(\gamma-1)^m \pmod {p^{u+v}}$.
  Let
  \begin{equation*}
   y_i=\begin{cases}
   \gamma-1,& i\leq k_0, \\
   0,& i>k_0,
   \end{cases}\ \ 1\leq i\leq p^{u+v}-1.
   \end{equation*}
  Then $y_i\in\{0,\gamma-1\}$ and $x-\sum\limits_{i=1}^{p^{u+v}-1}y_i^m=x-k_0(\gamma-1)^m\equiv\sum\limits_{i=1}^{p^{u}-1}x_{i,1}^m\pmod {p^{u+v}}$.
  Hence, (\ref{thm7.3}) is valid for $N=1$.

  Assume (\ref{thm7.3}) is valid for some $N\geq 1$.
  Next we prove (\ref{thm7.3}) for $N+1$.
  If $z\in \{x_{i,N}:1\leq i\leq p^u-1\}$, then we have $z\equiv \gamma-1 \pmod {p^u}$.
  So $p\nmid z$.
  Note that
  \begin{equation*}
    \left(z+(\gamma-1)\gamma^N\right)^m=z^m+mz^{m-1}(\gamma-1)\gamma^N+\sum_{j=2}^{m}\binom{m}{j}z^{m-j}(\gamma-1)^j \gamma^{jN}.
  \end{equation*}
  For $j\in\{2,3,\cdots,m\}$, since $p\geq 3$, we have
  \begin{equation*}
  \begin{cases}
   j\geq 2=v_p(j)+2, & \text{if}\ v_p(j)=0. \\
   j\geq p^{v_p(j)}\geq v_p(j)+2, & \text{if}\ v_p(j)\ge 1.
   \end{cases}
   \end{equation*}
   So $\left((j-1)N-1\right)u+v_p\left(\binom{m}{j}\right)\ge j-2+v_p(m)-v_p(j)+v_p\left(\binom{m-1}{j-1}\right)\geq v$.
   Then for $j\in\{2,3,\cdots,m\}$, we obtain $v_p\left(\binom{m}{j}z^{m-j}(\gamma-1)^j\gamma^{jN}\right)= v_p\left(\binom{m}{j}\right)+juN\geq v+u(N+1)$.
   Since $z^{m-1}\equiv (\gamma-1)^{m-1}\pmod {p^u}$, we can see that
   \begin{equation*}
     mz^{m-1}(\gamma-1)\gamma^N=m_1 z^{m-1}(\gamma-1)\gamma_1^N p^{v+uN}\equiv m_1 (\gamma-1)^{m}\gamma_1^N p^{v+uN}\pmod {p^{v+u(N+1)}}.
   \end{equation*}
   So
   \begin{equation*}
     \left(z+(\gamma-1)\gamma^N\right)^m\equiv z^m+m_1(\gamma-1)^{m}\gamma_1^N p^{v+uN}\pmod {p^{v+u(N+1)}}.
   \end{equation*}
  Using the fact that $p\nmid m_1(\gamma-1)^{m}\gamma_1^N$, we derive that the residue class of module $p^{u(N+1)+v}$ of the set $\left\{k m_1(\gamma-1)^{m}\gamma_1^N p^{uN+v}: k\in\{0,1,\cdots,p^u-1\}\right\}$
  is $\left\{k p^{uN+v}:k\in\{0,1,\cdots,p^u-1\}\right\}$.
  Then there is a $k_1\in\{0,1,\cdots,p^u-1\}$ such that $x-\sum\limits_{i=1}^{p^{u+v}-1}y_i^m\equiv k_1 m_1(\gamma-1)^{m}\gamma_1^N p^{uN+v}+\sum\limits_{i=1}^{p^{u}-1}x_{i,N}^m\pmod {p^{u(N+1)+v}}$.
  Let
  \begin{equation*}
  x_{i,N+1}=\begin{cases}
   x_{i,N},& i>k_1, \\
   x_{i,N}+(\gamma-1)\gamma^N,&  i\leq k_1.
   \end{cases}
   \end{equation*}
   It follows that $x_{i,N+1}\equiv x_{i,N} \equiv \gamma-1 \pmod {p^u}$ and
   \begin{equation*}
    \sum_{i=1}^{p^{u}-1}x_{i,N+1}^m\equiv k_1 m_1(\gamma-1)^{m}\gamma_1^N p^{uN+v}+\sum_{i=1}^{p^{u}-1}x_{i,N}^m\equiv x-\sum_{i=1}^{p^{u+v}-1}y_i^m\pmod {p^{u(N+1)+v}}.
  \end{equation*}
  This completes the proof of the claim.

  Proceeding as in the proof of the last theorem, we can see that  $x=\sum\limits_{i=1}^{p^{u+v}-1}y_i^m+\sum\limits_{i=1}^{p^u-1}x_i^m$,
  where $x_i=\lim\limits_{N\rightarrow\infty}x_{i,N}\in \mathcal{C}_\gamma$.
  Moreover, $p^{u+v}-1+p^u-1=\frac{1}{|\gamma|_p}\left(\frac{1}{|m|_p}+1\right)-2$.
  The proof is completed.
\end{proof}
Applying the above theorem with $p=\gamma=3$, we can obtain the following result.
\begin{corollary}
  The smallest positive integer $k$ that satisfies $\mathbb{Z}_3=\{x_1^2+x_2^2+\cdots+x_k^2:x_1,x_2,\cdots,x_k\in \mathcal{C}_3\}$ is $4$.
  In other words, we have  $\mathcal{G}_3(2)=4$.
\end{corollary}
\begin{proof}
  Let $\gamma=p=3$ and $m=2$.
  We have $\frac{1}{|\gamma|_p}\left(\frac{1}{|m|_p}+1\right)-2=4$.
  It follows from Theorem 7.3 that $\mathbb{Z}_3=\{x_1^2+x_2^2+x_3^2+x_4^2:x_1,x_2,x_3,x_4\in \mathcal{C}_3\}$.

  We now show that $\mathbb{Z}_3\neq\{x_1^2+x_2^2+x_3^2:x_1,x_2,x_3\in \mathcal{C}_3\}$.
  Note that $\mathcal{C}_3\subseteq\{0,2,6,8\}+3^2\mathbb{Z}_3$.
  We have $\{x^2:x\in \mathcal{C}_3\}\subseteq\{0,1,4\}+3^2\mathbb{Z}_3$.
  It follows that $\{x_1^2+x_2^2:x_1,x_2\in \mathcal{C}_3\}\subseteq\{0,1,2,4,5,8\}+3^2\mathbb{Z}_3$
  and
  \begin{equation*}
  \{x_1^2+x_2^2+x_3^2:x_1,x_2,x_3\in \mathcal{C}_3\}\subseteq\{0,1,2,3,4,5,6,8\}+3^2\mathbb{Z}_3.
  \end{equation*}
   Consequently, $7\notin\{x_1^2+x_2^2+x_3^2:x_1,x_2,x_3\in \mathcal{C}_3\}$ and $\mathbb{Z}_3\neq\{x_1^2+x_2^2+x_3^2:x_1,x_2,x_3\in \mathcal{C}_3\}$.
\end{proof}

We end this section with the following problem.

\begin{problem}
For any prime $p$, $\gamma\in p\mathbb{Z}_p$ with $2|\gamma|_p<1$ and integer $m\ge 2$, how to find the smallest positive integer $k$ that satisfies
\begin{equation*}
  \mathbb{Z}_p=\{x_1^m+x_2^m+\cdots+x_k^m:x_1,x_2,\cdots,x_k\in \mathcal{C}_\gamma\}?
\end{equation*}
\end{problem}


\end{document}